\documentclass[12pt,draft,leqno]{article}
\usepackage{amssymb,amsmath,amsfonts,amsthm}

\newtheorem{theorem}{Theorem}[section]
\newtheorem{lemma}{Lemma}[section]
\newtheorem{proposition}{Proposition}[section]

\newtheorem{corollary}{Corollary}[section]
\newtheorem{definition}{Definition}[section]
\newtheorem{remark}{Remark}[section]

\begin{document}
\title{Inhomogeneous microlocal propagation of singularities in Fourier Lebesgue spaces}
\author{Gianluca Garello, Alessandro Morando}
\date{ \ }
\maketitle
\begin{abstract}
Some results of microlocal continuity for pseudodifferential operators whose non regular symbols belong to weighted Fourier Lebesgue spaces are given.\\
Inhomogeneous local and microlocal propagation of singularities of Fourier Lebesgue type are then  studied, with applications to some classes of semilinear equations.
\end{abstract}
\section{Introduction}\label{INT}
Consider the general nonlinear partial differential equation
\begin{equation}\label{int_eq_1}
F\left(x, \{\partial ^\alpha u\}_{\vert \alpha\vert \leq m}\right)=0,
\end{equation} 
where $F(x,\zeta)\in C^\infty (\mathbb R^n\times \mathbb C^N)$ for suitable positive integer $N$.\\
In order to investigate local and microlocal regularity of the solutions, it is quite natural to reduce the study to the linearized equation, obtained by differentiation with respect to $x_j$
\begin{equation}\label{int_eq_2}   
\sum_{\vert \alpha\vert\leq m}\frac{\partial F}{\partial \zeta^\alpha} \left(x, \{\partial^\beta u\}_{\vert \beta\vert\leq m}\right) \partial ^\alpha\partial_{x_j}u=-\frac{\partial F}{\partial x_j}\left(x, \{\partial^\beta u\}_{\vert \beta\vert\leq m}\right).
\end{equation}
Notice that the regularity of the  coefficients $a_\alpha(x)=\dfrac{\partial F}{\partial \zeta^\alpha} \left(x, \{\partial^\beta u\}_{\vert \beta\vert\leq m}\right)$, depends on the  solution $u$ and the function $F(x,\zeta)$. We need then to study as first step the algebra properties in the function spaces in which we are intended to operate, as well as the behaviour of the pseudodifferential operators with symbols in such spaces.
\\
When working in H\"older spaces and Sobolev spaces $H^{s,2}$, we can refer  to the paradifferential calculus, developed by J.M. Bony and Y. Meyer, \cite{BO1}, \cite{ME1}, \cite{TA1}. Generalizations of these arguments to the symbols of quasi homogeneous, or completely inhomogeneous type can be found  in \cite{YA1}, \cite{YA2}, \cite{G}, \cite{G0}, \cite{G3}.\\
In this paper we fix the attention on pseudodifferential operators with symbols in weighted Fourier Lebesgue spaces $\mathcal F L^p_\omega(\mathbb R^n)$, following an idea of S. Pilipovi\'c - N. Teofanov - J. Toft, \cite{PTT2}, \cite{PTT1}.\\ Passing now to consider  the microlocal regularity properties, let us notice that the H\"ormander wave front set, introduced in \cite{HOR1} for smooth singularities and extended to the Sobolev spaces $H^{s,2}$ in \cite{HOR_hyp}, uses as basic tool the conic neighborhoods in $\mathbb R^n\setminus\{0\}$. 
Thus the homogeneity properties  of the symbol $p(x,\xi)$ and the characteristic set Char $P$ of the (pseudo) differential operator  $P=p(x,D)$, play a key role. In order to better adapt the study to a wider class of equations, starting from the fundamental papers of R. Beals \cite{BE1}, L. H\"ormander \cite{HO3}, an extensive literature about weighted pseudodifferential operators has been developed, see e.g. \cite{YA1}, \cite{CM1}, \cite{LR1}, \cite{BBR}.\\
 We are particularly interested here in the generalizations of the wave front set  not involving the use of conic neighborhoods and consequently the homogeneity properties. In some cases, for example in the study of propagation of singularities of the Schroedinger operator, $i\partial_t-\Delta$, we can use the quasi-homogeneous wave front set, introduced in \cite{LA1}, see further  \cite{SA1}, \cite{YA2}. More generally, failing of any homogeneity properties, the propagation of the  microlocal singularities are described in terms of filter of neighborhoods, introduced in \cite{RO1} and further developed in \cite{GA}, \cite{G}, \cite{G0}, \cite{G1}, \cite{G3}, \cite{GM2.1}, \cite{GM2.2} \cite{GM3}.\\
 In  some previous works of the authors continuity and microlocal properties are considered in Sobolev spaces in $L^p$ setting, see \cite{GM1}, \cite{GM2}, \cite{GMqh2}, \cite{GM2.1}, \cite{GM2.2}, \cite{GM3}.\\
 In the present paper we prove a result of propagation of singularities of Fourier Lebesgue type, for partial (pseudo) differential equations, whose symbol satisfies generalized elliptic properties. Namely we obtain an extension of the well known propagation of singularities given by H\"ormander \cite{HOR_hyp} for the Sobolev wave front set $WF_{H^{s,2}}$ and operators of order $m$:
 \begin{equation*}
  WF_{H^{s-m,2}}(Pf)\subset WF_{H^{s,2}}(f)\subset WF_{H^{s-m,2}}(Pf)\cup {\rm Char}\, (P),
\end{equation*}   
 given in terms of filter of microlocal singularities and quasi-homogeneous wave front set.\\
 Applications to semilinear partial differential equations are given at the end.
 The plan of the paper is the following: in \S \ref{prel_sct} the weight funtions $\omega$ and the Fourier Lebesgue spaces $\mathcal F L^p_\omega(\mathbb R^n)$ are introduced and their properties studied. In \S \ref{algebra_sct}, \S \ref{w_FL_pdo_sct}, under suitable additional conditions on the weight function, algebra properties in $\mathcal F L^p_\omega(\mathbb R^n)$ and continuity of pseudodifferential operators with symbols in Fourier Lebesgue spaces are studied. The microlocal regularity, in terms of inhomogeneous neighborhoods, is introduced and studied. In \S \ref{micro_w_FL_sct} the microlcal properties of Fourier Lebesgue spaces are defined, while the propagation of microlocal singularities is given in \S \ref {appl_sct}, namely in Proposition \ref{filter_inclusions}. In \S \ref {appl_sct}  applications to semilinear equations are studied, with specific examples in the field of quasi-homogeneous partial differential equations.

\section{Preliminaries}\label{prel_sct}
\subsection{Weight functions}\label{wf_sct}
Throughout the paper, we call {\it weight function} any positive measurable map $\omega:\mathbb R^n\rightarrow ]0,+\infty[$ satisfying the following {\it temperance} condition
\begin{equation*}
(\mathcal T)\qquad\qquad\qquad\omega(\xi)\le C(1+\vert\xi-\eta\vert)^N\omega(\eta)\,,\qquad\forall\,\xi\,,\eta\in\mathbb R^n\,,\qquad\qquad\qquad\,\,\,
\end{equation*}
for suitable positive constants $C$ and $N$.

In the current literature, a positive function $\omega$ obeying condition ($\mathcal T$) is said to be either {\it temperated} (see  \cite{G3}, \cite{HOR1}) or, in the field of Modulation Spaces, {\it polynomially moderated} (cf. \cite{F1}, \cite{GRO1},  \cite{PTT1}, \cite{PTT2}). 

For $\omega, \omega_1$  weight functions; we write $\omega\preceq\omega_1$ to mean that, for some $C>0$
\begin{equation*}
\omega(x)\le C\omega_1(x)\,,\qquad\forall\,x\in\mathbb R^n\,;
\end{equation*}

moreover we say that  $\omega, \omega_1$ are {\it equivalent}, writing $\omega \asymp\omega_1$ in this case, if
\begin{equation}\label{equiv}
\omega\preceq\omega_1\qquad\mbox{and}\qquad\omega_1\preceq\omega\,.
\end{equation}

Applying ($\mathcal T$) it yields at once that
$\omega(\xi)\le C(1+\vert\xi\vert)^N\omega(0)\,$ and
$\omega(0)\le C(1+\vert-\xi\vert)^N\omega(\xi)=C(1+\vert\xi\vert)^N\omega(\xi)\,$, for any $\xi\in\mathbb R^n$.

Thus, for every weight function $\omega$ there exist constants $C\ge 1$ and $N>0$ such that
\begin{equation}\label{pg_cond}
\frac1{C}(1+\vert\xi\vert)^{-N}\le\omega(\xi)\le C(1+\vert\xi\vert)^N\,,\qquad\forall\,\xi\in\mathbb R^n\,.
\end{equation}
\begin{proposition}\label{pm_prop1}
Let $\omega, \omega_1$ be two weight functions and $s\in\mathbb R$. Then $\omega\omega_1$, $1/\omega$ and $\omega^s$ are again weight functions.
\end{proposition}
\begin{proof}
Assume that for suitable constants $C, C_1, N, N_1$
\begin{equation}\label{PMomega1}
\begin{split}
&\omega(\xi)\le C(1+\vert\xi-\eta\vert)^N\omega(\eta)\,,\\
&\omega_1(\xi)\le C_1(1+\vert\xi-\eta\vert)^{N_1}\omega(\eta)\qquad\forall\,\xi\,,\eta\in\mathbb R^n\,.
\end{split}
\end{equation}
Then we deduce that
\begin{equation*}
\begin{split}
&\omega(\xi)\omega_1(\xi)\le CC_1(1+\vert\xi-\eta\vert)^{N+N_1}\omega(\eta)\omega_1(\eta)\,,\\
&1/\omega(\eta)\le C(1+\vert\eta-\xi\vert)^N 1/\omega(\xi)\,,\qquad\forall\,\xi\,,\eta\in\mathbb R^n\,,
\end{split}
\end{equation*}
which show that $\omega\omega_1$ and $1/\omega$ are temperate.

If $s\ge 0$ then condition ($\mathcal{T}$) for $\omega^s$ follows at once from \eqref{PMomega1}. If $s<0$ it suffices to observe that $\omega^s=1/\omega^{-s}$ and then combining the preceding results.
\end{proof}
We introduce now some further onditions on the weight function $\omega$ which will be repeatedly used in the following.
\begin{itemize}
\item[($\mathcal{SV}$)] {\bf Slowly varying condition}: there exist positive constants $C\ge 1$, $N$ such that
\begin{equation}\label{sv_cond}
\frac1{C}\le\frac{\omega(\eta)}{\omega(\xi)}\le C\,,\qquad\mbox{when}\,\,\vert\eta-\xi\vert\le\frac1{C}\omega(\xi)^{1/N}\,;
\end{equation}
\item[($\mathcal{SA}$)] {\bf Sub additive condition}: for some positive constant $C$
\begin{equation}\label{ring_cond}
\omega(\xi)\le C\left\{\omega(\xi-\eta)+\omega(\eta)\right\}\,,\qquad\forall\,\xi\,,\eta\in\mathbb R^n\,;
\end{equation}
\item[($\mathcal{SM}$)] {\bf Sub multiplicative condition}: for some positive constant $C$
\begin{equation}\label{subm_cond}
\omega(\xi)\le C\omega(\xi-\eta)\omega(\eta)\,,\qquad\forall\,\xi\,,\eta\in\mathbb R^n\,;
\end{equation}
\item[($\mathcal{G}$)] {\bf $\delta$ condition}: for some positive constants $C$ and $0<\delta<1$
\begin{equation}\label{delta_cond}
\omega(\xi)\le C\left\{\omega(\eta)\omega(\xi-\eta)^\delta+\omega(\eta)^\delta\omega(\xi-\eta)\right\}\,,\qquad\forall\,\xi\,,\eta\in\mathbb R^n\,;
\end{equation}
\item[($\mathcal{B}$)] {\bf Beurling's condition}: for some positive constant $C$
\begin{equation}\label{B_cond}
\sup\limits_{\xi\in\mathbb R^n}\int_{\mathbb R^n}\frac{\omega(\xi)}{\omega(\xi-\eta)\omega(\eta)}\,d\eta\le C\,.
\end{equation}
\end{itemize}

For a thorough account on the relations between the properties introduced above, we refer to \cite{G3}. For reader's convenience, here we quote and prove only the following result.

\begin{proposition}\label{wf_relationship}
For the previous conditions the following relationships are true.
\begin{itemize}
\item[i.] Assume that $\omega$ is uniformly bounded from below in $\mathbb R^n$, that is
\begin{equation}\label{bb_cond}
\inf_{\xi\in\mathbb R^n}\omega(\xi)=c >0\,.
\end{equation}
Then
$(\mathcal{SV})\,\,\Rightarrow\,\,(\mathcal{T})$ and $(\mathcal{G})\,\,\Rightarrow\,\,(\mathcal{SM})$.
\item[ii.] Assume that
\begin{equation}\label{integrability}
\frac{1}{\omega}\in L^1(\mathbb R^n)\,.
\end{equation}
Then $(\mathcal{SA})\,\,\Rightarrow\,\,(\mathcal{B})$ and $(\mathcal{G})\,\,\Rightarrow\,\,\omega^{\frac1{1-\delta}}\,\,\mbox{satisfies}\,\,(\mathcal{B})$.
\end{itemize}
\end{proposition}
\begin{proof}
{\it Statement i:} Let the constants $C$, $N$ be fixed as in \eqref{sv_cond}. For $\xi, \eta\in\mathbb R^n$ such that $\vert\xi-\eta\vert\le\frac1{C}\omega(\xi)^{1/N}$, it follows directly from \eqref{sv_cond} that $\omega(\xi)\le C\omega(\eta)\le C\omega(\eta)(1+\vert\xi-\eta\vert)^N
$. On the other hand, when $\vert\xi-\eta\vert>\frac1{C}\omega(\xi)^{1/N}$ from \eqref{bb_cond} we deduce at once
$\omega(\xi)\le C^N\vert\xi-\eta\vert^N\le\frac{C^N}{c}\omega(\eta)(1+\vert\xi-\eta\vert)^N\,.
$

This shows the validity of the first implication. As for the second one, it is sufficient to observe that $\omega(\xi)\ge\varepsilon>0$ and $0<\delta<1$ yield at once
\begin{equation}\label{stima_delta}
\omega(\xi)^\delta\le c^{\delta-1} \omega(\xi)\,, \qquad\forall\,\xi\in\mathbb R^n\,.
\end{equation}
Then the result follows from estimating by \eqref{stima_delta} the function $\omega^\delta$ in the right-hand side of \eqref{delta_cond}.

{\it Statement ii:} For every $\xi\in\mathbb R^n$, using \eqref{ring_cond} we get
\begin{equation*}
\frac{\omega(\xi)}{\omega(\xi-\eta)\omega(\eta)}\le C\left\{\frac{1}{\omega(\eta)}+\frac{1}{\omega(\xi-\eta)}\right\}\,;
\end{equation*}
hence the first implication follows observing that, by a suitable change of variables, the right-hand side is an integrable function on $\mathbb R^n$, whose integral is independent of $\xi$.

Concerning the second implication, for every $\xi\in\mathbb R^n$, from \eqref{delta_cond} we get
\begin{equation*}
\begin{split}
\left(\frac{\omega(\xi)}{\omega(\xi-\eta)\omega(\eta)}\right)^{\frac1{1-\delta}}&\le C_\delta\left(\omega(\xi-\eta)^{\delta-1}+\omega(\eta)^{\delta-1}\right)^{\frac1{1-\delta}}\\
&\le C_\delta\left\{\left(\omega(\xi-\eta)^{\delta-1}\right)^{\frac1{1-\delta}}+\left(\omega(\eta)^{\delta-1}\right)^{\frac1{1-\delta}}\right\}\\
&=C_\delta\left\{\frac1{\omega(\xi-\eta)}+\frac1{\omega(\eta)}\right\}\,,
\end{split}
\end{equation*}
for a suitable constant $C_\delta>0$ depending on $\delta$. Now we conclude as in the proof of the first implication.

\end{proof}
{\it Examples}
\begin{itemize}
\item[1.] The {\it standard homogeneous weight}
\begin{equation}\label{ell_wf}
\langle\xi\rangle^m:=\left(1+\vert\xi\vert^2\right)^{m/2}\,,\qquad\xi\in\mathbb R^n\,,\,\,\,m\in\mathbb R\,,
\end{equation}
is a weight function according to the definition given at the beginning of this section.

The well-known Peetre inequality
\begin{equation}\label{peetre}
\langle\xi\rangle^m\le 2^{\vert m\vert}\langle\xi-\eta\rangle^{\vert m\vert}\langle\eta\rangle^m\,,\qquad\forall\,\xi\,,\eta\in\mathbb R^n\,,
\end{equation}
shows that $\langle\cdot\rangle^m$ satisfies the condition ($\mathcal{T}$) for every $m\in\mathbb R$ (with $N=\vert m\vert$) as well as the condition ($\mathcal{SM}$) for $m\ge 0$. For every $m\ge 0$, the function $\langle\cdot\rangle^m$ also fulfils ($\mathcal{SV}$) (where $N=m$) as a consequence of a Taylor expansion, and ($\mathcal{SA}$).  Finally $1/\langle\cdot\rangle^m$ satisfies the integrability condition \eqref{integrability} as long as $m>n$; hence $\langle\cdot\rangle^m$ satisfies condition ($\mathcal{B}$) for $m>n$, in view of the statement ii of Proposition \ref{wf_relationship}.

\item[2.] For $M=(m_1,\dots,m_n)\in\mathbb N^n$, the {\it quasi-homogeneous weight} is defined as
\begin{equation}\label{quasi_ell_wf}
\langle\xi\rangle_M:=\left(1+\sum\limits_{j=1}^n\xi_j^{2m_j}\right)^{1/2}\,,\qquad\forall\,\xi\in\mathbb R^n\,.
\end{equation}

The quasi-homogeneous weight obeys the  polynomial growth condition
\begin{equation}\label{pg_q_hom}
\frac1{C}\langle\xi\rangle^{m_\ast}\le \langle\xi\rangle_M\le C\langle\xi\rangle^{m^\ast}\,,\qquad\forall\,\xi\in\mathbb R^n\,,
\end{equation}
for some positive constant $C$ and  $m_\ast:=\min\limits_{1\le j\le n}m_j$, $m^\ast:=\max\limits_{1\le j\le n}m_j$. Moreover, for all $s\in\mathbb R$, the derivatives of $\langle\cdot\rangle_M^s$ decay according to the estimates below
\begin{equation}\label{q_hom_der}
\left\vert\partial^\alpha_\xi\langle\xi\rangle_M^s\right\vert\le C_\alpha\langle\xi\rangle_M^{s-\langle\alpha\,,\frac1{M}\rangle}\,,\qquad\forall\,\xi\in\mathbb R^n\,,\, \forall \alpha\in\mathbb Z^n_+,
\end{equation}
where $\langle\alpha\,, \,\frac1{M}\rangle:=\sum\limits_{j=1}^n\frac{\alpha_j}{m_j}$ and $C_\alpha>0$ is a suitable constant. Using \eqref{q_hom_der} with $s=\frac1{m^\ast}$ we may prove that $\langle\cdot\rangle_M$ fulfils condition ($\mathcal{SV}$) with $N=m^\ast$; indeed from the trivial identities
\begin{equation}\label{taylor1}
\langle\xi\rangle_M^{1/m^\ast}-\langle\eta\rangle_M^{1/m^\ast}\!=\sum\limits_{j=1}^n(\xi_j-\eta_j)\int_0^1\partial_j\left(\langle\cdot\rangle_M^{1/m^\ast}\right)(\eta+t(\xi-\eta))\,dt
\end{equation}
and \eqref{q_hom_der}, we deduce
\begin{equation*}\label{sv_q_hom1}
\begin{split}
\left\vert\langle\xi\rangle_M^{1/m^\ast}\right.&\left.-\langle\eta\rangle_M^{1/m^\ast}\right\vert\le\sum\limits_{j=1}^n C_j\vert\xi_j-\eta_j\vert\int_0^1\langle \eta+t(\xi-\eta)\rangle_M^{1/m^\ast-1/m_j}\,dt\\
&\le\sum\limits_{j=1}^n C_j\vert\xi_j-\eta_j\vert\,,
\end{split}
\end{equation*}
since $m^\ast\ge m_j$ for every $j$. Now for $\xi$, $\eta$ satisfying
$
\vert\xi-\eta\vert\le\varepsilon \langle\xi\rangle_M^{1/m^\ast}
$,
from the previous inequality we deduce
$
\left\vert\langle\xi\rangle_M^{1/m^\ast}-\langle\eta\rangle_M^{1/m^\ast}\right\vert\!\le\!\widehat C \varepsilon \langle\xi\rangle_M^{1/m^\ast}
$,
with $\widehat C:=\sum\limits_{j=1}^n C_j$, that is
$
(1-\widehat C \varepsilon)\langle\xi\rangle_M^{1/m^\ast}\le\langle\eta\rangle_M^{1/m^\ast}\le(1+\widehat C \varepsilon)\langle\xi\rangle_M^{1/m^\ast}
$,
from which we get the conclusion, if we assume for instance $0<\varepsilon\le\frac1{2\widehat C}$.

From ($\mathcal{SV}$) and the trivial inequality $\langle\xi\rangle_M\ge 1$, using the statement i of Proposition \ref{wf_relationship} we obtain that ($\mathcal{T}$) is also satisfied with $N=m^\ast$.

Also, the weight $\langle\cdot\rangle_M$ satisfies condition ($\mathcal{SA}$) and, because of the left inequality in \eqref{pg_q_hom}, $1/\langle\cdot\rangle_M^s$ satisfies condition \eqref{integrability} provided that $s>\frac{n}{m_\ast}$. Then from the statement ii of Proposition \ref{wf_relationship}, $\langle\cdot\rangle_M^s$ satisfies condition ($\mathcal{B}$) for $s>\frac{n}{m_\ast}$.

At the end, let us observe that for $M=(m,\dots,m)$, with a given $m\in\mathbb N$, $\langle\xi\rangle_M \asymp\langle\xi\rangle^m$.

\item[3.] Let $\mathcal P$ be a {\it complete polyhedron} of $\mathbb{R}^n$ in the sense of Volevich-Gindikin, \cite{VG2}. The {\em multi-quasi-elliptic weight functions} is defined by
\begin{equation}\label{mqe_wf}
\lambda_{\cal P}(\xi):=\left(\sum_{\alpha\in V(\mathcal P)}\xi^{2\alpha}\right)^{1/2},\qquad\xi\in\mathbb R^n\,,
\end{equation}
where $V(\mathcal P)$ denotes the set of {\it vertices} of $\mathcal P$.

We recall that a {\em convex polyhedron} $\mathcal{P}\subset\mathbb{R}^n$ is the convex hull of a finite set $V(\mathcal{P})\subset \mathbb{R}^n$ of convex-linearly independent points, called {\em vertices} of $\mathcal{P}$, and univocally determined by $\mathcal {P}$ itself. Moreover, if $\mathcal{P}$ has non empty interior, it is completely described by
\begin{displaymath}\label{WF6}
\mathcal{P}=\{\xi\in\mathbb{R}^n;\nu\cdot\xi\geq 0, \forall \nu\in \mathcal{N}_0(\mathcal{P})\}\cap\{\xi\in\mathbb{R}^n;\nu\cdot\xi\leq 1,\forall\nu\in\mathcal{N}_1(\mathcal{P})\};
\end{displaymath}
where $\mathcal{N}_0(\mathcal{P})\subset \{\nu\in \mathbb{R}^n;\vert\nu\vert =1\}$, $\mathcal{N}_1(\mathcal{P})\subset \mathbb{R}^n$ are finite sets univocally determined by $\mathcal{P}$ and, as usual, $\nu\cdot\xi=\sum_{j=1}^n\nu_j\xi_j$. The boundary of $\mathcal{P}$, $\mathcal {F}(\mathcal{P})$, is made of faces $\mathcal{F}_{\nu}(\mathcal{P})$ which are the convex hulls of the vertices of $\mathcal{P}$ lying on the hyper-planes $H_\nu$ orthogonal to $\nu\in \mathcal{N}_0(\mathcal{P})\cup\mathcal{N}_1(\mathcal{P})$, of equation :
\begin{equation*}
\nu\cdot\xi=0\quad \mbox{if}\quad \nu\in\mathcal{N}_0(\mathcal{P}),\quad\quad \nu\cdot\xi=1\quad\mbox{if}\quad \nu\in\mathcal{N}_1(\mathcal{P}).
\end{equation*}
A {\it complete polyhedron} is a convex polyhedron $\mathcal{P}\subset\mathbb{R}^n_+=:\{\xi\in\mathbb{R}^n:\ \xi_j\geq 0, \ j=1,...,n\}$ such that:
\begin{itemize}
\item[i)] $V(\mathcal{P})\subset\mathbb{N}^n$;
\item[ii)] $(0,\dots,0)\in V(\mathcal{P})$, and $V(\mathcal{P})\neq\{(0,\dots,0)\}$;
\item[iii)] $\mathcal{N}_0(\mathcal{P})=\{e_1,\dots,e_n\}$ with $e_j=(0,\dots, 1_{j-entry},\dots 0)\in\mathbb{R}^n_+$;
\item[iv)] every $\nu\in \mathcal{N}_1(\mathcal{P})$ has strictly positive components $\nu_j$, $j=1,...,n$.
\end{itemize}

On can prove that the multi-quasi-elliptic weight growths at infinity according to the following estimates
\begin{equation}\label{pg_mqe}
\frac1{C}\langle\xi\rangle^{\mu_0}\le \lambda_{\mathcal P}(\xi)\le C\langle\xi\rangle^{\mu_1}\,,\qquad\forall\,\xi\in\mathbb R^n\,,
\end{equation}
for a suitable positive constant $C$ and where
\begin{equation}\label{min_max_order}
\mu_0:=\min\limits_{\gamma\in V(\mathcal P)\setminus\{0\}}\vert\gamma\vert\quad\mbox{and}\quad\mu_1:=\max\limits_{\gamma\in V(\mathcal P)}\vert\gamma\vert
\end{equation}
are called {\it minimum} and {\it maximum order} of $\mathcal P$ respectively. Moreover it can be proved  that for all $s\in\mathbb R$ the derivatives of $\lambda_{\mathcal P}^s$ decay according to the estimates below
\begin{equation}\label{mqe_der}
\left\vert\partial^\alpha_\xi\lambda_{\mathcal P}^s(\xi)\right\vert\le C_\alpha\lambda_{\mathcal P}(\xi)^{s-\frac1{\mu}\vert\alpha\vert}\,,\qquad\forall\,\xi\in\mathbb R^n\,,
\end{equation}
where
\begin{equation}\label{formal_order}
\mu:=\max\left\{1/\nu_j\,,\,\,j=1,\dots,n\,,\,\,\nu\in\mathcal N_1(\mathcal P)\right\}\,,
\end{equation}
satisfying $\mu\ge\mu_1$, is the so-called {\it formal order} of $\mathcal P$, see e.g. \cite{BBR} (see also \cite{GM1} where more general decaying estimates for $\lambda_{\mathcal P}^s$ are established).

Representing $\lambda_{\mathcal P}(\xi)^{1/\mu}-\lambda_{\mathcal P}(\eta)^{1/\mu}$ as in \eqref{taylor1}, for arbitrary $\xi, \eta\in\mathbb R^n$, and using \eqref{mqe_der} (with $s=1/\mu$), we deduce that $\lambda_{\mathcal P}$ satisfies ($\mathcal{SV}$) with $N=\mu$. Using also that $\lambda_{\mathcal P}(\xi)\ge 1$ (recall that $0\in V(\mathcal P)$), in view of Proposition \ref{wf_relationship} it follows that $\lambda_{\mathcal P}$ also satisfies ($\mathcal{T}$) with $N=\mu$, hence it is a weight function agreeing to the definition given at the beginning of this section.

The weight function $\lambda_{\mathcal P}$ does not satisfy condition ($\mathcal{SA}$); on the other hand it can be shown (see \cite{G0}) that condition ($\mathcal G$) is verified taking
\begin{equation}\label{delta}
\delta=\max\limits_{\beta\in\mathcal P\setminus\mathcal F(\mathcal P)}\max\limits_{\nu\in\mathcal N_1(\mathcal P)}\left\{\nu\cdot\beta\right\}\,.
\end{equation}
Since, from the left inequality in \eqref{pg_mqe} we also derive that $\lambda_{\mathcal P}^{-s}$ satisfies \eqref{integrability} for $s>\frac{n}{\mu_0}$, we conclude from the statement ii of Proposition \ref{wf_relationship} that $\lambda_{\mathcal P}^{r}$ satisfies condition ($\mathcal{B}$) if $r>\frac{n}{(1-\delta)\mu_0}$ for $\delta$ defined above. 

In the end, we notice that $\lambda_{\mathcal P}$ verifies ($\mathcal{SM}$) as a consequence of ($\mathcal{G}$), since $\lambda_{\mathcal P}(\xi)\ge 1$ and $0<\delta<1$, cf. Proposition \ref{wf_relationship}, statement i.
\end{itemize}

\begin{remark}\label{rmk:5}
{\rm We notice that the quasi-homogeneous weight $\langle\cdot\rangle_M$, $M=(m_1,\dots,m_n)\in\mathbb N^n$, considered in the example 2 is just the multi-quasi-elliptic weight $\lambda_{\mathcal P}$ introduced in the example 3 corresponding to the complete polyhedron $\mathcal P$ defined by the convex hull of the finite set $V(\mathcal P)=\{0\,,\,\, m_je_j\,,\,\,\,j=1,\dots,n\}$; in particular, the growth estimates \eqref{pg_q_hom} are the particular case of \eqref{pg_mqe} corresponding to the previous polyhedron $\mathcal P$ (in which case $\mu_0=m_\ast$ and $\mu_1=m^\ast$). Notice however that the decaying estimates \eqref{q_hom_der} satisfied by the quasi-homogeneous weight $\langle\cdot\rangle_M$ do not admit a counterpart in the case of the general multi-quasi-elliptic weight $\lambda_{\mathcal P}$. Estimates \eqref{q_hom_der} give a precise decay in each coordinate direction: the decrease of $\langle\xi\rangle_M$ corresponding to one derivative with respect to $\xi_j$ is measured by $\langle\xi\rangle_M^{-1/m_j}$\footnote{In other words the decay of the derivatives is measured here by the {\it vector weight} $(\langle\xi\rangle_M^{1/m_1},\dots,\langle\xi\rangle_M^{1/m_n})$, in the sense of {\it vector weighted} symbol classes, see \cite{G}, \cite{GM3}, \cite{RO1}.}. The lack of homogeneity in the weight associated to a general complete polyhedron $\mathcal P$ in \eqref{mqe_wf}, prevents from extending to derivatives of $\lambda_{\mathcal P}(\xi)$ the  decay properties in \eqref{q_hom_der}: estimates \eqref{mqe_der} do not take account of the decay corresponding {\it separately to each coordinate direction}.}
\end{remark}

\subsection{Weighted Lebesgue and Fourier--Lebesgue spaces}\label{FL_spaces_sct}
Let $\omega:\mathbb R^n\rightarrow ]0,+\infty[$ be a weight function.
\begin{definition}\label{w_L_spaces}
For every $p\in[1,+\infty]$, the weighted Lebesgue space $L^p_{\omega}(\mathbb R^n)$ is defined as the set of the (equivalence classes of) measurable functions $f:\mathbb R^n\rightarrow\mathbb C$ such that
\begin{equation}\label{w_L_cond}
\begin{split}
&\int_{\mathbb R^n}\omega(x)^p\vert f(x)\vert^p\,dx<+\infty\,,\quad\mbox{if}\,\,p<+\infty\,,\\
&\\
&\omega f\,\,\mbox{is essentially bounded in}\,\,\mathbb R^n\,,\quad\mbox{if}\,\,p=+\infty\,.
\end{split}
\end{equation}
\end{definition}
For every $p\in[1,+\infty]$, $L^p_{\omega}(\mathbb R^n)$ is Banach space with respect to the natural norm
\begin{equation}\label{norma_w_L}
\Vert f\Vert_{L^p_{\omega}}:=
\begin{cases}
\left(\int_{\mathbb R^n}\omega(x)^p\vert f(x)\vert^p\,dx\right)^{1/p}\,,\quad\mbox{if}\,\,p<+\infty\,,\\
\\
{\rm ess\,sup}_{x\in\mathbb R^n}\omega(x)\vert f(x)\vert\,,\quad\mbox{if}\,\,p=+\infty\,.
\end{cases}
\end{equation}

\begin{remark}\label{rmk:1}
{\rm It is easy to see that for all $p\in[1,+\infty]$
\begin{equation}\label{imbedding1}
L^p_{\omega_2}(\mathbb R^n)\hookrightarrow L^p_{\omega_1}(\mathbb R^n)\,,\qquad\mbox{if}\,\,\omega_1\preceq\omega_2\,.
\end{equation}
If in particular $\omega_1\asymp\omega_2$ then $L^p_{\omega_1}(\mathbb R^n)\equiv L^p_{\omega_2}(\mathbb R^n)$, and the norms defined in \eqref{norma_w_L} corresponding to $\omega_1$ and $\omega_2$  are equivalent. When the weight function $\omega$ is constant the related weighted space $L^p_{\omega}(\mathbb R^n)$ reduces to the standard Lebesgue space of order $p$, denoted as usual by $L^p(\mathbb R^n)$.}
\end{remark}

\begin{remark}\label{rmk:2}
{\rm For an arbitrary $p\in[1,+\infty]$, $f\in L^p_{\omega}(\mathbb R^n)$ and every $\varphi\in\mathcal S(\mathbb R^n)$ we obtain
\begin{equation}\label{holder}
\int_{\mathbb R^n}f(x)\varphi(x)\,dx\le\left\Vert\frac{\varphi}{\omega}\right\Vert_{L^q}\Vert\omega f\Vert_{L^p}\,,\qquad\frac1{p}+\frac1{q}=1\,.
\end{equation}
From \eqref{holder} and the estimates \eqref{pg_cond}, we deduce at once that
\begin{equation*}
\mathcal S(\mathbb R^n)\hookrightarrow L^p_{\omega}(\mathbb R^n)\hookrightarrow\mathcal S^\prime(\mathbb R^n)\,.
\end{equation*}
Moreover, $C^\infty_0(\mathbb R^n)$ is a dense subspace of $L^p_{\omega}(\mathbb R^n)$ when $p<+\infty$, see \cite{G2}.}
\end{remark}

\begin{remark}\label{rmk:8}
{\rm For $\Omega$ open subset of $\mathbb R^n$,  $L^p_{\omega}(\Omega)$, for any $p\in[1,+\infty]$, is the set of (equivalence classes of) measurable functions on $\Omega$ such that
\begin{equation}\label{Lpomega}
\Vert f\Vert^p_{L^p_{\omega}(\Omega)}:=\int_{\Omega}\omega(x)^p\vert f(x)\vert^p\,dx<+\infty
\end{equation}
(obvious modification for $p=+\infty$). $L^p_{\omega}(\Omega)$ is endowed with a structure of Banach space with respect to the natural norm defined by \eqref{Lpomega}.}
\end{remark}

\begin{definition}\label{w_FL_spaces}
For every $p\in[1,+\infty]$ and $\omega(\xi)$ temperate weight function, the weighted Fourier Lebesgue space $\mathcal F L^p_{\omega}(\mathbb R^n)$ is the vector space of all distributions $f\in\mathcal S^\prime(\mathbb R^n)$ such that
\begin{equation}\label{w_FL_cond}
\widehat f\in L^p_{\omega}(\mathbb R^n)\,,
\end{equation}
equipped with the natural norm
\begin{equation}\label{w_FL_norm}
\Vert f\Vert_{\mathcal FL^p_{\omega}}:=\Vert\widehat f\Vert_{L^p_{\omega}}\,.
\end{equation}
\end{definition}
Here $\widehat f$ is the Fourier transform $\widehat f(\xi)=\int e^{-i\xi\cdot x}f(x)\,dx$, defined in $\mathcal S(\mathbb R^n)$ and extended to  $\mathcal S'(\mathbb R^n)$.

The spaces $\mathcal FL^p_{\omega}(\mathbb R^n)$ were  introduced in H\"ormander \cite{HOR1}, with the notation $\mathcal B_{p,k}$, $k(\xi)$ weight function, for the study of the regularity of solutions to hypoelliptic partial differential equations with constant coefficients, see also \cite{G2}, \cite{G1}.

From the mapping properties of the Fourier transform on $\mathcal S(\mathbb R^n)$ and $\mathcal S^\prime(\mathbb R^n)$ and the above stated properties of weighted Lebesgue spaces we can conclude, see again \cite{G2}, that for all $p\in[1,+\infty]$ and $\omega$ temperate weight function
\begin{itemize}
\item[(a)] $\mathcal F L^p_{\omega}(\mathbb R^n)$ is a Banach space with respect to the norm \eqref{w_FL_norm};
\item[(b)] $\mathcal S(\mathbb R^n)\hookrightarrow \mathcal F L^p_{\omega}(\mathbb R^n)\hookrightarrow\mathcal S^\prime(\mathbb R^n)$;
\item[(c)] $C^\infty_0(\mathbb R^n)$ is a dense subspace of $\mathcal F L^p_{\omega}(\mathbb R^n)$ when $p<+\infty$;
\item[(d)] $\mathcal F L^p_{\omega_2}(\mathbb R^n)\hookrightarrow\mathcal F L^p_{\omega_1}(\mathbb R^n)$ if $\omega_1\preceq\omega_2$; in particular, we obtain that $\mathcal F L^p_{\omega_2}(\mathbb R^n)\equiv\mathcal F L^p_{\omega_1}(\mathbb R^n)$ as long as $\omega_1\asymp\omega_2$ and the norms corresponding to $\omega_1$ and $\omega_2$ by \eqref{w_FL_norm} are equivalent in this case.
\end{itemize}
When $\omega$ is a positive constant the weighted space $\mathcal F L^p_{\omega}(\mathbb R^n)$ is simply denoted by $\mathcal F L^p(\mathbb R^n)$. Moreover we will adopt the shortcut notations $L^p_s(\mathbb R^n):= L^p_{\langle\cdot\rangle^s}(\mathbb R^n)$, $\mathcal F L^p_s(\mathbb R^n):=\mathcal F L^p_{\langle\cdot\rangle^s}(\mathbb R^n)$ for the corresponding Lebesgue and Four\-ier Lebesgue spaces.

Analogously, when $\omega(\xi)=\langle\xi\rangle_M^s$ or $\omega(\xi)=\lambda_{\mathcal P}(\xi)^s$, for $s\in\mathbb R$, the corresponding Lebesgue and Fourier Lebesgue spaces will be denoted $L^p_{s,M}(\mathbb R^n)$, $\mathcal F L^p_{s,M}(\mathbb R^n)$ and $L^p_{s,\mathcal P}(\mathbb R^n)$, $\mathcal F L^p_{s, \mathcal P}(\mathbb R^n)$ respectively.

\vspace{.5cm}
A local counterpart of Fourier Lebesgue spaces can be introduced in the following natural way (see \cite{PTT2}).

\begin{definition}\label{loc_w_FL_spaces}
For $\omega$  weight function, $\Omega$ open subset of $\mathbb R^n$ and any $p\in[1,+\infty]$, $\mathcal FL^p_{\omega,\,{\rm loc}}(\Omega)$ is the class of all distributions $f\in\mathcal D^\prime(\Omega)$ such that $\varphi f\in\mathcal FL^p_{\omega}(\mathbb R^n)$ for every $\varphi\in C^\infty_0(\Omega)$.

For $x_0\in\Omega$, $f\in\mathcal FL^p_{\omega,\,{\rm loc}}(x_0)$ if there exists $\phi\in C^\infty_0(\Omega)$, with $\phi(x_0)\neq 0$, such that $\phi f\in\mathcal FL^p_\omega(\mathbb R^n)$.
\end{definition}

The family of semi-norms
\begin{equation}\label{loc_w_FL_norm}
f\mapsto\Vert \varphi f\Vert_{\mathcal F L^p_{\omega}}\,,\qquad\varphi\in C^\infty_0(\Omega)\,,
\end{equation}
provides $\mathcal FL^p_{\omega\,,{\rm loc}}(\Omega)$ with a natural Fr\'echet space topology. Moreover the following inclusions hold true with continuous embedding
\begin{equation}\label{inclusioni1}
C^\infty(\Omega)\hookrightarrow\mathcal F L^p_{\omega\,,{\rm loc}}(\Omega)\hookrightarrow\mathcal D^\prime(\Omega)
\end{equation}
and for $\Omega_1\subset\Omega_2$ open sets
\begin{equation}\label{inclusioni2}
\mathcal F L^p_{\omega}(\mathbb R^n)\hookrightarrow\mathcal F L^p_{\omega\,,{\rm loc}}(\Omega_2)\hookrightarrow \mathcal F L^p_{\omega\,,{\rm loc}}(\Omega_1)\,.
\end{equation}
\begin{remark}\label{rmk:3}
{\rm It is worth noticing that, as it was proved in \cite{RSTT}, {\it locally} the weighted Fourier Lebesgue spaces $\mathcal F L^q_{\omega}(\mathbb R^n)$ are the same as the weighted {\it modulation spaces} $M^{p,q}_{\omega}(\mathbb R^n)$ and the {\it Wiener amalgam spaces}  $W^{p,q}_{\omega}(\mathbb R^n)$, in the sense that
\begin{equation*}
\mathcal F L^q_{\omega}(\mathbb R^n)\cap\mathcal E^\prime(\mathbb R^n)=M^{p,q}_{\omega}(\mathbb R^n)\cap\mathcal E^\prime(\mathbb R^n)=W^{p,q}_{\omega}(\mathbb R^n)\cap\mathcal E^\prime(\mathbb R^n)\,.
\end{equation*}
We refer to Feichtinger \cite{F1} and Gr\"ochenig \cite{GRO1} for the definition and basic properties of modulation and amalgam spaces.}
\end{remark}

\vspace{.5cm}
Agreeing with the previous notations, when the weight function $\omega$ reduces to those considered in the examples 1, 2, 3 above, the corresponding local Fourier Lebesgue spaces will be denoted respectively by $\mathcal F L^p_{s,{\rm loc}}(\Omega)$, $\mathcal F L^p_{s,M,{\rm loc}}(\Omega)$, $\mathcal F L^p_{s,\mathcal P,{\rm loc}}(\Omega)$.

\vspace{.5cm}
Notice at the end that, from Plancherel Theorem, when $p=2$ the global and local weighted Fourier Lebesgue spaces $\mathcal F L^2_{\omega}(\mathbb R^n)$, $\mathcal F L^2_{\omega\,,{\rm loc}}(\Omega)$ coincide with weighted spaces of Sobolev type, see Garello \cite{G3} for an extensive study of such spaces.

\section{Algebra conditions in spaces $\mathcal FL^p_{\omega}(\mathbb  R^n)$}\label{algebra_sct}
In order to seek conditions on the weight function $\omega$ which allow the Fourier Lebesgue space $\mathcal F L^p_{\omega}(\mathbb R^n)$ to be an algebra with respect to the point-wise product, let us first state a general continuity result in the framework of suitable mixed-norm spaces of Lebesgue type.

Following \cite{F1}, \cite{GRO1} and in particurar\cite{PTT2}, for $p, q\in [1,+\infty]$ we denote by $\mathcal L_1^{p,q}(\mathbb R^{2n})$ the space of all (equivalence classes of) measurable functions $F=F(\zeta,\eta)$ in $\mathbb R^{n}\times\mathbb R^n$ such that the mixed norm
\begin{equation}\label{mixed_norm1}
\Vert F\Vert_{\mathcal L_1^{p,q}}:=\left(\int\left(\int\vert F(\zeta,\eta)\vert^p\,d\zeta\right)^{q/p}d\eta\right)^{1/q}
\end{equation}
is finite (with obvious modifications if $p$ or $q$ equal $+\infty$).

We also define $\mathcal L_2^{p,q}(\mathbb R^{2n})$ to be the space of measurable functions $F=F(\xi,\eta)$ in $\mathbb R^{n}\times\mathbb R^n$ such that the norm
\begin{equation}\label{mixed_norm2}
\Vert F\Vert_{\mathcal L_2^{p,q}}:=\left(\int\left(\int\vert F(\zeta,\eta)\vert^q\,d\eta\right)^{p/q}d\zeta\right)^{1/p}
\end{equation}
is finite.


\begin{lemma}\label{lemma_cont2}
For $p, q\in [1,+\infty]$ satisfying $\frac1{p}+\frac1{q}=1$, let $f=f(\zeta,\eta)\in\mathcal L^{p,\infty}_1(\mathbb R^{2n})$ and $F=F(\zeta,\eta)\in\mathcal L_2^{\infty,q}(\mathbb R^{2n})$. Then the linear map
\begin{equation}\label{TF2}
\begin{split}
T: C_0^\infty(\mathbb R^n)&\rightarrow\mathcal S^\prime(\mathbb R^n)\\
&g\mapsto Tg:=\int F(\xi,\eta)f(\xi-\eta,\eta)g(\eta)\,d\eta\,.
\end{split}
\end{equation}
extends uniquely to a continuous map from $L^p(\mathbb R^n)$ into itself, still denoted by $T$; moreover its operator norm is bounded as follows
\begin{equation}\label{estimate_TF}
\Vert  T\Vert_{\mathcal L(L^p)}\le \Vert f\Vert_{\mathcal L_1^{p,\infty}}\Vert F\Vert_{\mathcal L_2^{\infty,q}}\,.
\end{equation}
\end{lemma}

The proof is given in \cite[Lemma 2.1]{G2}, where the statement reads in quite different formulation. The reader can find a restricted version, independently proved, in \cite[Proposition 3.2]{PTT2}.

\begin{proposition}\label{algebra_prop}
Assume that $\omega$, $\omega_1$, $\omega_2$ are weight functions such that
\begin{equation}\label{B_Lq}
C_q:=\sup\limits_{\xi\in\mathbb R^n}\left\Vert\frac{\omega(\xi)}{\omega_1(\xi-\cdot)\omega_2(\cdot)}\right\Vert_{L^q}<+\infty\,,
\end{equation}
for some $q\in[1,+\infty]$, and let $p\in [1,+\infty]$ be the conjugate exponent of $q$. Then
\begin{itemize}
\item[(1)] the point-wise product map $(f_1,f_2)\mapsto f_1 f_2$ from $\mathcal S(\mathbb R^n)\times\mathcal S(\mathbb R^n)$ to $\mathcal S(\mathbb R^n)$ extends uniquely to a continuous bilinear map from $\mathcal F L^p_{\omega_1}(\mathbb R^n)\times \mathcal F L^p_{\omega_2}(\mathbb R^n)$ to $\mathcal F L^p_{\omega}(\mathbb R^n)$. Moreover for all $f_i\in\mathcal F L^p_{\omega_i}(\mathbb  R^n)$, $i=1,2$, the following holds:
    \begin{equation}\label{prod_est_1}
    \Vert f_1 f_2\Vert_{\mathcal F L^p_{\omega}}\le C_q\Vert f_1\Vert_{\mathcal F L^p_{\omega_1}}\Vert f_2\Vert_{\mathcal F L^p_{\omega_2}}\,.
    \end{equation}
\item[(2)] for every open set $\Omega\subseteq\mathbb R^n$ the point-wise product map $(f_1,f_2)\mapsto f_1 f_2$ from $C^\infty_0(\Omega)\times C^\infty_0(\Omega)$ to $C^\infty_0(\Omega)$ extends uniquely to a continuous bilinear map from $\mathcal F L^p_{\omega_1,{\rm loc}}(\Omega)\times \mathcal F L^p_{\omega_2,{\rm loc}}(\Omega)$ to $\mathcal F L^p_{\omega,{\rm loc}}(\Omega)$.
\end{itemize}
\end{proposition}
\begin{proof}
The proof of statement (2) follows at once from that of statement (1).

As for the proof of statement (1), for given $f_1, f_2\in\mathcal S(\mathbb R^n)$ one easily computes:
\begin{equation}\label{representation}
\begin{split}
\omega(\xi)\widehat{f_1 f_2}(\xi)&=(2\pi)^{-n}\int\omega(\xi)\widehat{f_1}(\xi-\eta)\widehat{f_2}(\eta)\,d\eta\\
&=\int F(\xi,\eta)f(\xi-\eta)g(\eta)\,d\eta\,,
\end{split}
\end{equation}
where
\begin{equation*}
F(\zeta,\eta)=\frac{\omega(\zeta)}{\omega_1(\zeta-\eta)\omega_2(\eta)}\,,\,\,\, f(\zeta)=\omega_1(\zeta)\widehat{f_1}(\zeta)\,,\,\,\, g(\zeta)=\omega_2(\zeta)\widehat{f_2}(\zeta)\,.
\end{equation*}
The right-hand side of \eqref{representation} provides a representation of $\omega\widehat{f_1 f_2}$ as an integral operator of the form \eqref{TF2}. Condition \eqref{B_Lq} just means that the function $F(\zeta,\eta)\in\mathcal L_2^{\infty, q}(\mathbb R^{2n})$ (cf. \eqref{mixed_norm2}) and of course the $\eta-$independent function $f=f(\zeta)\in\mathcal S(\mathbb R^n)$ also belongs to $\mathcal L^{p,\infty}_1(\mathbb R^{2n})$. Then applying to \eqref{representation} the result of Lemma \ref{lemma_cont2}, together with the definition of the norm in Fourier Lebesgue spaces, we obtain that the point-wise product $f_1f_2$ satisfies the estimates in \eqref{prod_est_1} and the proof is concluded.
\end{proof}

When the weight functions $\omega$, $\omega_1$ and $\omega_2$ in the statement of Proposition \ref{algebra_prop} coincide, condition \eqref{B_Lq} provides a sufficient condition for $\mathcal F L^p_{\omega}(\mathbb R^n)$ (or its localized counterpart $\mathcal F L^p_{\omega, {\rm loc}}(\Omega)$) is an algebra for the point-wise product. Then we have the following

\begin{corollary}\label{algebra_prop2}
Let $\omega$ be a weight function such that
\begin{equation}\label{B_q}
C_q:=\sup\limits_{\xi\in\mathbb R^n}\left\Vert\frac{\omega(\xi)}{\omega(\xi-\cdot)\omega(\cdot)}\right\Vert_{L^q}<+\infty\,,
\end{equation}
for $q\in[1,+\infty]$, and $p\in[1,+\infty]$ the conjugate exponent of q. Then
\begin{itemize}
\item[(1)] $\left(\mathcal F L^p_{\omega}(\mathbb R^n),\cdot\right)$ is an algebra and for $f_1, f_2\in \mathcal F L^p_{\omega}(\mathbb R^n)$
    \begin{equation}\label{prod_est_2}
    \Vert f_1 f_2\Vert_{\mathcal F L^p_{\omega}}\le C_q\Vert f_1\Vert_{\mathcal F L^p_{\omega}}\Vert f_2\Vert_{\mathcal F L^p_{\omega}}\,.
    \end{equation}
\item[(2)] for every open set $\Omega\subseteq\mathbb R^n$, $\left(\mathcal F L^p_{\omega,{\rm loc}}(\Omega),\cdot\right)$ is an algebra.
\end{itemize}
\end{corollary}

The algebra properties of Corollary \ref{algebra_prop2} let us handle the composition of a Fourier Lebesgue distribution with an entire analytic functions; namely we have the following result, see \cite[Corollary 2.1]{G2}.

\begin{corollary}\label{composition_prop}
Under the same assumptions of Corollary \ref{algebra_prop2} on $\omega$ and $p$, let $F:\mathbb{C}\rightarrow \mathbb{C}$ be an entire analytic function such that $F(0)=0$. Then $F(u)\in\mathcal FL^p_{\omega}(\mathbb R^n)$ for every $u\in \mathcal FL^p_{\omega}(\mathbb R^n)$, and
\begin{equation}\label{LPE61}
\Vert F(u)\Vert_{\mathcal FL^{p}_{\omega}}\leq C\Vert u\Vert_{\mathcal FL^{p}_{\omega}}\,,
\end{equation}
with $C=C(p,F,\Vert u\Vert_{\mathcal FL^{p}_{\omega}})$.
\end{corollary}

\begin{remark}\label{rmk:3.1}
{\rm A counterpart of Corollary \ref{composition_prop} for the local space $\mathcal FL^p_{\omega, \rm loc}(\Omega)$  can be obtained by replacing $F=F(z)$ above with a function $F=F(x,\zeta)$ mapping $\Omega\times {\mathbb C}^M$ into $\mathbb{C}$, which is locally smooth with respect to the real variable $x\in\Omega$ and entire analytic in the complex variable $\zeta\in\mathbb{C}^M$ uniformly on compact subsets of $\Omega$; namely:
$$
F(x,\zeta)=\sum_{\beta\in\mathbb{Z}^M_+}c_{\beta}(x)\zeta^{\beta}, \quad \quad c_{\beta}\in C^{\infty}(\Omega), \ \zeta\in \mathbb{C}^M\,,
$$
where, for any compact set $K\subset\Omega$, $\alpha\in\mathbb{Z}^n_+$, $\beta\in\mathbb{Z}^M_+$, $\sup_{x\in K}\vert \partial_x^{\alpha}c_{\beta}(x)\vert\leq c_{\alpha,K}\lambda_{\beta}$ and $F_1(\zeta):=\sum_{\beta\in\mathbb{Z}_+^M}\lambda_{\beta}\zeta^{\beta}$ is entire analytic.

Under the assumptions of Corollaries \ref{algebra_prop2}, \ref{composition_prop}, we have that $F(x,u)\in \mathcal FL^{p}_{\omega,{\rm loc}}(\Omega)$ as long as the components of the vector $u=(u_1,...,u_M)$ belong to $\mathcal FL^{p}_{\omega, {\rm loc}}(\Omega)$.}
\end{remark}

\begin{remark}\label{rmk:4}
{\rm Let us notice that for $1\le q<+\infty$, condition \eqref{B_q} on $\omega$ is nothing but condition ($\mathcal B$) for the weight function $\omega^q$, while for $q=+\infty$ \eqref{B_q} reduces to condition ($\mathcal{SM}$) for $\omega$. The latter case means that $\left(\mathcal FL^1_{\omega}(\mathbb R^n),\cdot\right)$ is an algebra provided that the weight function $\omega$ is sub-multiplicative, which is in agreement with the more general result of \cite[Lemma 1.6]{PTT2}.}
\end{remark}

The next result shows that the {\it sub-multiplicative} condition $(\mathcal{SM})$ on a weight function is a necessary condition for the corresponding scale of weighted Fourier Lebesgue spaces to possess the algebra property.

\begin{proposition}\label{prop_sm}
Let $\omega_1$, $\omega_2$, $\omega$ be weight functions and $p_1, p_2, p\!\in\![1,\!+\infty]$. If we assume that the map $(f_1,f_2)\mapsto f_1f_2$ from $\mathcal S(\mathbb R^n)\times\mathcal S(\mathbb R^n)$ to $\mathcal S(\mathbb R^n)$ extends uniquely to a continuous bilinear map from $\mathcal F L^{p_1}_{\omega_1}(\mathbb R^n)\times\mathcal F L^{p_2}_{\omega_2}(\mathbb R^n)$ to $\mathcal F L^{p}_{\omega}(\mathbb R^n)$, then a positive constant $C$ exists such that
\begin{equation}\label{stima_algebra1}
\omega(\eta+\theta)\le C\omega_1(\eta)\omega_2(\theta)\,,\qquad\forall\,\eta\,,\theta\in\mathbb R^n\,.
\end{equation}
\end{proposition}
\begin{proof}
By assumption, there exists a constant $C^\prime>0$ such that for all $f\in \mathcal F L^{p_1}_{\omega_1}(\mathbb R^n)$, $g\in\mathcal F L^{p_2}_{\omega_2}(\mathbb R^n)$
\begin{equation}\label{stima_algebra2}
\Vert f g\Vert_{\mathcal F L^p_{\omega}}\le C^\prime\Vert f\Vert_{\mathcal F L^{p_1}_{\omega_1}}\Vert g\Vert_{\mathcal F L^{p_2}_{\omega_2}}\,.
\end{equation}
From condition ($\mathcal{T}$) (cf. also \eqref{pg_cond}) we may find some constants $\varepsilon>0$, $C>0$ such that
\begin{equation}\label{pm1}
\varepsilon\le\frac{\omega_i(\eta)}{\omega_i(\xi)}\le\varepsilon^{-1}\,\,(i=1,2)\,\,\mbox{and}\,\,\varepsilon\le\frac{\omega(\eta)}{\omega(\xi)}\le\varepsilon^{-1}\,,\quad\mbox{when}\,\,\vert\xi-\eta\vert\le\frac{\varepsilon}{C}
\end{equation}
We follow here the same arguments of the proof of \cite[Theorem 3.8]{G3}. Let us take a function $\varphi\in\mathcal S(\mathbb R^n)$ such that
\begin{equation*}
\widehat{\varphi}(\xi)\ge 0\qquad\mbox{and}\qquad{\rm supp}\,\widehat\varphi\subseteq\left\{\xi\in\mathbb R^n\,:\,\,\vert\xi\vert\le\frac{\varepsilon}{2C}\right\}\,.
\end{equation*}
For arbitrary points $\eta, \theta\in\mathbb R^n$, let us define
\begin{equation}\label{def_f_g}
f(x)=e^{i\eta\cdot x}\varphi(x)\,,\qquad g(x)=e^{i\theta\cdot x}\varphi(x)\,,
\end{equation}
hence
\begin{equation*}
f(x)g(x)=e^{i(\eta+\theta)\cdot x}\varphi^2(x)\,.
\end{equation*}
In view of the assumption on the support of $\widehat\varphi$ we compute:
\begin{equation}\label{norma_feg}
\begin{split}
&\Vert f\Vert^{p_1}_{\mathcal F L^{p_1}_{\omega_1}}=\int\omega_1(\xi)^{p_1}\widehat{\varphi}(\xi-\eta)^{p_1}d\xi=\int_{\vert\xi-\eta\vert\le\frac{\varepsilon}{2C}}\omega_1(\xi)^{p_1}\widehat{\varphi}(\xi-\eta)^{p_1}d\xi\,,\\
&\Vert g\Vert^{p_2}_{\mathcal F L^{p_2}_{\omega_2}}=\int\omega_2(\xi)^{p_2}\widehat{\varphi}(\xi-\theta)^{p_2}d\xi=\int_{\vert\xi-\theta\vert\le\frac{\varepsilon}{2C}}\omega_2(\xi)^{p_2}\widehat{\varphi}(\xi-\theta)^{p_2}d\xi\,,
\end{split}
\end{equation}
In the domain of the integrals above \eqref{pm1} holds, then we get
\begin{equation*}
\Vert f\Vert^{p_1}_{\mathcal F L^{p_1}_{\omega_1}}\le\varepsilon^{-p_1}\omega_1(\eta)^{p_1}\int_{\vert\xi-\eta\vert\le\frac{\varepsilon}{C}}\widehat{\varphi}(\xi-\eta)^{p_1}d\xi=c_{1}^{p_1}\varepsilon^{-p_1}\omega_1(\eta)^{p_1}\,,
\end{equation*}
hence
\begin{equation}\label{stima_norma_f}
\Vert f\Vert_{\mathcal F L^{p_1}_{\omega_1}}\le c_{1}\varepsilon^{-1}\omega_1(\eta)\,,
\end{equation}
where $c_{1}:=\Vert\widehat\varphi\Vert_{L^{p_1}}$. The same holds true for the norm of $g$ in $\mathcal FL^{p_2}_{\omega_2}(\mathbb R^n)$, by replacing $\eta$ with $\theta$, that is
\begin{equation}\label{stima_norma_g}
\Vert g\Vert_{\mathcal F L^{p_2}_{\omega_2}}\le c_2\varepsilon^{-1}\omega_2(\theta)\,,\quad\mbox{for}\,\,c_2=\Vert\widehat\varphi\Vert_{L^{p_2}}\,.
\end{equation}
The preceding calculations are performed under the assumption that both $p_1$ and $p_2$ are finite; however the same estimates \eqref{stima_norma_f}, \eqref{stima_norma_g} can be easily extended to the case when $p_1$ or $p_2$ equals $+\infty$.

As for the norm in $\mathcal F L^p_{\omega}(\mathbb R^n)$ of $fg$ we compute
\begin{equation}\label{norma_fg}
\Vert fg\Vert^{p}_{\mathcal F L^{p}_{\omega}}=\int_{\vert\xi-\eta-\theta\vert\le\frac{\varepsilon}{C}}\omega(\xi)^{p}\widehat{\varphi^2}(\xi-\eta-\theta)^{p}d\xi\,,
\end{equation}
where we used that ${\rm supp}\,\widehat{\varphi^2}={\rm supp}\,(\widehat\varphi\ast\widehat\varphi)\subseteq\left\{\vert\xi\vert\le\frac{\varepsilon}{C}\right\}$ and it is assumed $p<+\infty$ (to fix the ideas). Then recalling again that \eqref{pm1} holds true for $\omega$ on ${\rm supp}\,\widehat{\varphi^2}$ we obtain
\begin{equation}\label{stima_norma_fg}
\Vert fg\Vert_{\mathcal F L^{p}_{\omega}}\ge c\varepsilon\omega(\eta+\theta)\,,\quad\mbox{with}\,\,c:=\Vert\widehat{\varphi^2}\Vert_{L^p}\,.
\end{equation}
The same estimate \eqref{stima_norma_fg} can be easily recovered in the case $p=+\infty$.

We use now \eqref{stima_norma_f}, \eqref{stima_norma_g} and \eqref{stima_norma_fg} to estimate the right- and left-hand sides of \eqref{stima_algebra2} written for $f$ and $g$ defined in \eqref{def_f_g} to get
\begin{equation*}
c\varepsilon\omega(\eta+\theta)\le C^\prime c_{1}c_2\varepsilon^{-2}\omega_1(\eta)\omega_2(\theta)\,.
\end{equation*}
In view of the arbitrariness of $\eta$, $\theta$ and since the constants $c_1$, $c_2$ and $\varepsilon$ are independent of $\eta$ and $\theta$ the preceding inequality gives \eqref{stima_algebra1} with $C=C^\prime c_{1}c_2c^{-1}\varepsilon^{-3}$.
\end{proof}

\begin{remark}\label{rmk:7}
{\rm It is worth observing that any specific relation is assumed on the exponents $p_1, p_2, p\in [1,+\infty]$ in the statement of Proposition \ref{prop_sm}. Notice also that condition \eqref{stima_algebra1} is just \eqref{B_Lq} for $q=+\infty$. When in particular $\omega_1=\omega_2=\omega$, it reduces to ($\mathcal{SM}$).

Notice also that from the results given by Corollary \ref{algebra_prop2} (see also Remark \ref{rmk:6}) and Proposition \ref{prop_sm}, we derive that condition ($\mathcal{SM}$) is {\it necessary and sufficient} to make the Fourier Lebesgue space $\mathcal F L^1_{\omega}(\mathbb R^n)$ an algebra for the point-wise product\footnote{However condition ($\mathcal{SM}$) is far from being sufficient for $\mathcal F L^p_{\omega}(\mathbb R^n)$ to be an algebra, as long as $p>1$.}.
}
\end{remark}

Combining the results of Corollary \ref{algebra_prop2} with the remarks made about the weight functions quoted in the examples 1-3 at the end of Section \ref{wf_sct} we can easily prove the following result.

\begin{corollary}\label{algebra_prop3}
Let $r\in\mathbb R$, $M=(m_1,\dots,m_n)\in\mathbb N^n$ and $\mathcal P$ a complete polyhedron of $\mathbb R^n$ be given and assume that $p,q\in[1,+\infty]$ satisfy $\frac1{p}+\frac1{q}=1$. Then
\begin{itemize}
\item[i.] $\left(\mathcal F L^p_{r}(\mathbb R^n)\,,\,\cdot\right)$ is an algebra if $r>\frac{n}{q}$;
\item[ii.] $\left(\mathcal F L^p_{r,M}(\mathbb R^n)\,,\,\cdot\right)$ is an algebra if $r>\frac{n}{m_\ast q}$, where $m_\ast=\min\limits_{1\le j\le n}m_j$;
\item[iii.] $\left(\mathcal F L^p_{r,\mathcal P}(\mathbb R^n)\,,\,\cdot\right)$ is an algebra if $r>\frac{n}{(1-\delta)\mu_0 q}$, where $\mu_0$ and $\delta$ are defined in \eqref{min_max_order} and \eqref{delta}.
\end{itemize}
Analogous statements hold true for the localized version of the previous spaces on an open subset $\Omega$ of $\mathbb R^n$, defined according to Definition \ref{loc_w_FL_spaces}.
\end{corollary}
\begin{proof}
Let us prove the statement iii of the Theorem; the proof of the other statements is completely analogous. Assume that $p>1$, thus $q<+\infty$. From \eqref{pg_mqe} we have that $s>\frac{n}{\mu_0 q}$ implies $\lambda_{\mathcal P}^{-sq}\in L^1(\mathbb R^n)$; on the other hand, $\lambda_{\mathcal P}^{sq}$ satisfies ($\mathcal{G}$) with $\delta$ defined as in \eqref{delta}, see Example 3 in Sect. \ref{prel_sct}. From Proposition \ref{wf_relationship} applied to $\lambda_{\mathcal P}^{sq}$ we derive that $\lambda_{\mathcal P}^{rq}$ fulfils condition ($\mathcal{B}$), which amounts to say that $\lambda_{\mathcal P}^{r}$ satisfies \eqref{B_q}, where $r=\frac{s}{1-\delta}$. Then the result of Corollary \ref{algebra_prop2} applies to $\mathcal FL^p_{\lambda_{\mathcal P}^{r}}(\mathbb R^n)=\mathcal F L^p_{r, \mathcal P}(\mathbb R^n)$ and gives the statement iii. Notice that condition $r>\frac{n}{(1-\delta)\mu_0 q}$ reduces to $r>0$ when $q=+\infty$ (corresponding to $p=1$). That $\mathcal F L^1_{r,\mathcal P}(\mathbb R^n)$ with $r>0$ is an algebra for the point-wise product follows again from Corollary \ref{algebra_prop2} by observing that $\lambda_{\mathcal P}^r$ satisfies ($\mathcal{SM}$) (that is \eqref{B_q} for $q=+\infty$).
\end{proof}

\begin{remark}\label{rmk:6}
{\rm In agreement with the observation made at the end of Section \ref{FL_spaces_sct}, for $p=2$ the lower bounds of $r$ given in i-iii of Corollary \ref{algebra_prop3} are exactly the same required to ensure the algebra property for the corresponding weighted Sobolev spaces (see \cite{G3} and \cite{GM2} for the case of a general $1<p<+\infty$).}
\end{remark}

To the end of this section, let us observe that, as a byproduct of Propositions \ref{algebra_prop} and \ref{prop_sm}, the following result can be proved.

\begin{proposition}\label{algebra_prop4}
Assume that $\omega$, $\omega_1$, $\omega_2$ are temperate weight functions satisfying condition \eqref{B_Lq} for some $1\le q<+\infty$. Then $\omega$, $\omega_1$, $\omega_2$ also satisfy condition \eqref{stima_algebra1}. In particular, if $\omega$ is a temperate weight function satisfying condition \eqref{B_q} for some $1\le q<+\infty$ then it also satisfies condition ($\mathcal{SM}$).
\end{proposition}

\begin{remark}\label{rmk:10}
{\rm The second part of Proposition \ref{algebra_prop4} slightly improves the result of \cite[Proposition 2.4]{G3}, where the sub-multiplicative condition ($\mathcal{SM}$) was deduced from Beurling condition ($\mathcal{B}$) (corresponding to \eqref{B_q} with $q=1$) and conditions ($\mathcal{SV}$) and \eqref{bb_cond}; here $\omega$, $\omega_1$, $\omega_2$ are only required to satisfy condition ($\mathcal T$) (included in our definition of weight function), which is implied by ($\mathcal{SV}$) and \eqref{bb_cond} in view of Proposition \ref{wf_relationship}.i.}
\end{remark}

\section{Pseudodifferential operators with symbols in weighted Fourier Lebesgue spaces}\label{w_FL_pdo_sct}
This section is devoted to the study of a class of pseudodifferential operators whose symbols $a(x,\xi)$ have a finite regularity of weighted Fourier Lebesgue type with respect to $x$.

Let us first recall that, under the only assumption $a(x,\xi)\in\mathcal S^\prime(\mathbb R^{2n})$, the pseudodifferential operator defined by
\begin{equation}\label{pdo}
a(x,D)f=(2\pi)^{-n}\int e^{ix\cdot\xi} a(x,\xi)\widehat f(\xi)d\xi\,,\qquad f\in\mathcal S(\mathbb R^n)\,,
\end{equation}
maps continuously $\mathcal S(\mathbb R^n)$ to $\mathcal S^\prime(\mathbb R^n)$\footnote{The integral in the right-hand side of \eqref{pdo} must be understood here in a weak (distributional) sense.}. Similarly, if $\Omega$ is an open subset of $\mathbb R^n$ and $a(x,\xi)\in\mathcal D^\prime(\Omega\times\mathbb R^n)$ is such that $\varphi(x)a(x,\xi)\in\mathcal S^\prime(\mathbb R^n\times\mathbb R^n)$ for every $\varphi\in C^\infty_0(\Omega)$, then \eqref{pdo} defines a linear continuous operator from $\mathcal S(\mathbb R^n)$ to $\mathcal D^\prime(\Omega)$.

Let us also recall that, as a linear continuous operator from $C^\infty_0(\Omega)$ to $\mathcal D^\prime(\Omega)$, every pseudodifferential operator with symbol $a(x,\xi)\in\mathcal D^\prime(\Omega\times\mathbb R^n)$ admits a (uniquely defined) Schwartz kernel $\mathcal K_a(x,y)\in\mathcal D^\prime(\Omega\times\Omega)$ such that
\begin{equation*}
\langle a(x,D)\psi,\varphi\rangle=\langle\mathcal K_a,\varphi\otimes\psi\rangle\,,\quad\forall\,\varphi\,,\psi\in C^\infty_0(\Omega)\,.
\end{equation*}
The operator $a(x,D)$ is said to be {\it properly supported} on $\Omega$ when the support of $\mathcal K_a$ is a {\it proper} subset of $\Omega\times\Omega$, that is ${\rm supp}\,\mathcal K_a\cap(\Omega\times K)$ and ${\rm supp}\,\mathcal K_a\cap (K\times\Omega)$ are compact subsets of $\Omega\times\Omega$, for every compact set $K\subset\Omega$. It is well known that every properly supported pseudodifferential operator continuously maps $C^\infty_0(\Omega)$ into the space $\mathcal E^\prime(\Omega)$ of compactly supported distributions and it extends as a linear continuous operator from $C^\infty(\Omega)$ into $\mathcal D^\prime(\Omega)$. In particular for every function $\phi\in C^\infty_0(\Omega)$ another function $\tilde\phi\in C^\infty_0(\Omega)$ can be found in such a way that
\begin{equation}\label{eq_pr_supp}
\phi(x) a(x,D)u=\phi(x) a(x,D)(\tilde\phi u)\,,\quad \forall\,u\in C^\infty(\Omega)\,.
\end{equation}

Following \cite{G2}, we introduce some local and global classes of symbols with finite Fourier Lebesgue regularity.

\begin{definition}\label{symbols}
Let $\omega=\omega(\xi)$, $\gamma=\gamma(\xi)$ be arbitrary weight functions.
\begin{itemize}
\item[1.]
A distribution $a(x,\xi)\in\mathcal S^\prime(\mathbb R^{2n})$ is said to belong to the class $\mathcal F L^p_{\omega} S_{\gamma}$ if $\xi\mapsto a(\cdot,\xi)$ is a measurable $\mathcal F L^p_{\omega}(\mathbb R^n)-$valued function on $\mathbb R^n_\xi$ such that
\begin{equation}\label{symbols_condt1}
\left\Vert\frac{a(\cdot,\xi)}{\gamma(\xi)}\right\Vert_{\mathcal F L^p_{\omega}}\le C\,,\quad\forall\,\xi\in\mathbb R^n\,,
\end{equation}
for some constant $C>0$. More explicitly, the above means that
\begin{equation*}
\frac{\omega(\eta)\widehat{a(\cdot,\xi)}(\eta)}{\gamma(\xi)}\in L^p(\mathbb R^n_\eta)\,,
\end{equation*}
with norm uniformly bounded with respect to $\xi$.
\item[2.] We say that a distribution $a(x,\xi)\in\mathcal D^\prime(\Omega\times\mathbb R^n)$, where $\Omega$ is an open subset of $\mathbb R^n$, belongs to $\mathcal F L^p_{\omega}S_{\gamma}(\Omega)$ if $\phi(x)a(x,\xi)\in \mathcal F L^p_{\omega} S_{\gamma}$ for every $\phi\in C^\infty_0(\Omega)$ (which amounts to have that $a(\cdot,\xi)/\gamma(\xi)\in\mathcal F L^p_{\omega,{\rm loc}}(\Omega)$ uniformly in $\xi$).
\end{itemize}
\end{definition}

\begin{remark}\label{rmk:11}
{\rm For $\omega$, $\gamma$ as in Definition \ref{symbols}, $\mathcal F L^p_{\omega} S_{\gamma}$ is a Banach space with respect to the norm
\begin{equation}\label{symbol_norm1}
\Vert a\Vert_{\mathcal FL^p_{\omega}S_\gamma}:=\sup\limits_{\xi\in\mathbb R^n}\left\Vert\frac{a(\cdot,\xi)}{\gamma(\xi)}\right\Vert_{\mathcal F L^p_{\omega}}\,,
\end{equation}
while $\mathcal FL^p_{\omega}S_\gamma(\Omega)$ is a Fr\'echet space with respect to the family of semi-norms
\begin{equation}\label{symbol_norm2}
a\mapsto \Vert\phi a\Vert_{\mathcal FL^p_{\omega}S_\gamma}\,,\qquad\phi\in C^\infty_0(\Omega)\,.
\end{equation}}
\end{remark}

Let us point out that any assumption is made about the $\xi-$derivatives of the symbol $a(x,\xi)$ in the above definition: the weight function $\gamma$ only measures the $\xi-$decay at infinity of the symbol itself. It is clear that $\mathcal F L^p_{\omega} S_{\gamma_1}\equiv \mathcal F L^p_{\omega} S_{\gamma_2}$, whenever $\gamma_1\sim\gamma_2$ (the same applies of course to the corresponding local classes on an open set). When the weight function $\gamma$ is an arbitrary positive constant, the related symbol class $\mathcal F L^p_{\omega} S_{\gamma}$ will be simply denoted as $\mathcal F L^p_{\omega} S$, and its symbols (and related pseudo-differential operators) will be referred to as {\it zero-th order symbols} (and {\it zero-th order operators}). Finally, we notice that for every weight function $\omega=\omega(\xi)$ and $p\in[1,+\infty]$, the inclusion $\mathcal F L^p_{\omega}(\mathbb R^n)\subset\mathcal F L^p_{\omega} S$ holds true (the elements of $\mathcal FL^p_{\omega}(\mathbb R^n)$ being regarded as $\xi-$independent symbols).

\begin{proposition}\label{continuity1}
For $p\in[1,+\infty]$ let the weight functions $\omega$, $\omega_1$, $\omega_2$ and $\gamma$ satisfy
\begin{equation}\label{B_q1}
C_q:=\sup\limits_{\xi\in\mathbb R^n}\left\Vert\frac{\omega_2(\xi)\gamma(\cdot)}{\omega_1(\cdot)\omega(\xi-\cdot)}\right\Vert_{L^q}<+\infty\,,
\end{equation}
where $q$ is the conjugate exponent of $p$. Then the following hold true.
\begin{itemize}
\item[i.] For every $a(x,\xi)\in\mathcal FL^p_{\omega}S_\gamma$ the pseudodifferential operator $a(x,D)$ extends to a unique linear bounded operator
    \begin{equation*}
    a(x,D):\mathcal FL^p_{\omega_1}(\mathbb R^n)\rightarrow\mathcal FL^p_{\omega_2}(\mathbb R^n)\,.
    \end{equation*}
\item[ii.] For every $a(x,\xi)\in\mathcal FL^p_{\omega}S_{\gamma}(\Omega)$, with $\Omega$ open subset of $\mathbb R^n$, the pseudodifferential operator $a(x,D)$ extends to a unique linear bounded operator
    \begin{equation*}
    a(x,D):\mathcal FL^p_{\omega_1}(\mathbb R^n)\rightarrow\mathcal FL^p_{\omega_2,{\rm loc}}(\Omega)\,.
    \end{equation*}
    If in addition the pseudodifferential operator $a(x,D)$ is properly supported, then it extends to a linear bounded operator
    \begin{equation*}
    a(x,D):\mathcal FL^p_{\omega_1,{\rm loc}}(\Omega)\rightarrow\mathcal FL^p_{\omega_2,{\rm loc}}(\Omega)\,.
    \end{equation*}
\end{itemize}
\end{proposition}
\begin{proof}
The second part of the statement {\it ii} is an immediate consequence of the first one; indeed, since the operator $a(x,D)$ is properly supported, for every function $\phi\in C^\infty_0(\Omega)$ another function $\tilde\phi\in C^\infty_0(\Omega)$ can be chosen in such a way that $\phi a(\cdot,D)u=\phi a(\cdot,D)(\tilde\phi u)$, cf. \eqref{eq_pr_supp}.

The first part of the statement {\it ii} follows, in its turn, from the statement {\it i} by noticing that $\phi(x)a(x,\xi)\in \mathcal FL^p_{\omega}S_\gamma$ for every function $\phi\in C^\infty_0(\Omega)$ (cf. Definition \ref{symbols}).

As for the proof of the statement {\it i}, we first observe that for every $u\in\mathcal S(\mathbb R^n)$ one computes
\begin{equation}\label{representation1}
\widehat{a(\cdot,D)u}(\eta)=(2\pi)^{-n}\int \widehat{a}(\eta-\xi,\xi)\widehat{u}(\xi)\,d\xi\,,
\end{equation}
where $\widehat{a}(\eta,\xi):=\widehat{a(\cdot,\xi)}(\eta)$ denotes the partial Fourier transform of the symbol $a(x,\xi)$ with respect to $x$. From \eqref{representation1} we find the following integral representation
\begin{equation}\label{representation2}
\begin{split}
\omega_2(\eta)&\widehat{a(\cdot,D)u}(\eta)=(2\pi)^{-n}\int \omega_2(\eta)\widehat{a}(\eta-\xi,\xi)\widehat{u}(\xi)\,d\xi\\
&=(2\pi)^{-n}\int \frac{\omega_2(\eta)\gamma(\xi)}{\omega(\eta-\xi)\omega_1(\xi)}\,\frac{\omega(\eta-\xi)\widehat{a}(\eta-\xi,\xi)}{\gamma(\xi)}\,\omega_1(\xi)\widehat{u}(\xi)\,d\xi\\
&=(2\pi)^{-n}\int F(\eta,\xi)f(\eta-\xi,\xi)g(\xi)\,d\xi\,,
\end{split}
\end{equation}
where it is set
\begin{equation}\label{positions}
f(\zeta,\xi)=\frac{\omega(\zeta)\widehat{a}(\zeta,\xi)}{\gamma(\xi)}\,,\,\, F(\zeta,\xi)=\frac{\omega_2(\zeta)\gamma(\xi)}{\omega(\zeta-\xi)\omega_1(\xi)}\,,\,\, g(\xi)=\omega_1(\xi)\widehat{u}(\xi)\,.
\end{equation}
The assumptions of Proposition \ref{continuity1} (see \eqref{B_q1}, \eqref{symbols_condt1}) yield the following
\begin{equation*}
\sup\limits_{\zeta\in\mathbb R^n}\Vert F(\zeta,\cdot)\Vert_{L^q}=C_q\,,\quad\sup\limits_{\xi\in\mathbb R^n}\Vert f(\cdot,\xi)\Vert_{L^p}=\Vert a\Vert_{\mathcal FL^p_{\omega}S_\gamma}\,,\quad g\in L^p(\mathbb R^n)\,.
\end{equation*}
Now we apply to $\omega_2(\eta)\widehat{a(\cdot,D)u}(\eta)$, written as the integral operator in \eqref{representation2}, the result of Lemma \ref{lemma_cont2}. Then we have
\begin{equation*}
\begin{split}
\Vert a(\cdot,D)u\Vert_{\mathcal F L^p_{\omega_2}}&=\Vert\omega_2 \widehat{a(\cdot,D)u}\Vert_{L^p}\le(2\pi)^{-n}C_q \Vert a\Vert_{\mathcal FL^p_{\omega}S_\gamma}\Vert u\Vert_{\mathcal FL^p_{\omega_1}}\,.
\end{split}
\end{equation*}
\end{proof}

\begin{remark}\label{rmk:12}
{\rm It is clear that Proposition \ref{continuity1} provides a generalization of the result given by Proposition \ref{algebra_prop}; indeed the multiplication by a given function $v=v(x)\in \mathcal FL^p_{\omega}(\mathbb R^n)\subset \mathcal FL^p_{\omega}S$ can be thought to as the zero-th order pseudodifferential operator with $\xi-$independent symbol $a(x,\xi)=v(x)$, cf. Remark \ref{rmk:10}.}
\end{remark}

\section{Microlocal regularity in weighted Fourier Lebesgue spaces}\label{micro_w_FL_sct}
This section is devoted to introduce a microlocal counterpart of the weighted Fourier Lebesgue spaces presented in Section \ref{FL_spaces_sct} and to define corresponding classes of pseudodifferential operators, with finitely regular symbols, naturally acting on such spaces.

Because of the lack of homogeneity of a generic weight function $\omega=\omega(\xi)$, in order to perform a microlocal analysis in the framework of weighted Fourier Lebesgue spaces it is convenient to replace the usual conic neighborhoods (used in Pilipovi\'c et al. \cite{PTT1, PTT2}) by a suitable notion of {\it $\varepsilon-$neighborhood} of a set, modeled on the weight function itself, following the approach of Rodino \cite{RO1} and Garello \cite{G}.

In the following, let $\omega:\mathbb R^n\rightarrow ]0,+\infty[$ be a weight function satisfying the subadditivity condition ($\mathcal{SA}$) and

\medskip

($\mathcal{SH}$): for a suitable constant $C\ge 0$
\begin{equation}\label{q_h}
\omega(t\xi)\le C\omega(\xi)\,,\quad\forall\,\xi\in\mathbb R^n\,,\,\,\vert t\vert\le 1\,.
\end{equation}

Every weight function $\omega=\omega(\xi)$ satisfying ($\mathcal{SA}$) and ($\mathcal{SH}$) also obeys the following
\begin{equation}\label{strong_sv}
\frac1{C}\le\frac{\omega(\xi+t\theta)}{\omega(\xi)}\le C\,,\qquad\mbox{when}\,\,\omega(\theta)\le\frac1{C}\omega(\xi)\,,\,\,\vert t\vert\le 1\,,
\end{equation}
for a suitable constant $C>1$, cf.\cite{G}.

Throughout the whole section, the weight function $\omega=\omega(\xi)$ will be assumed to be continuous\footnote{This assumption is not as much restrictive, since it can be shown (see e.g. \cite{T1}) that for any weight function $\omega$ an equivalent weight function $\omega_0$ exists such that $\omega_0\in C^\infty(\mathbb R^n)$ and $\partial^\alpha\omega_0/\omega_0\in L^\infty(\mathbb R^n)$ for each multi-index $\alpha$.}. Then an easy consequence of condition ($\mathcal{SH}$) is that $\omega(\xi)$ satisfies \eqref{bb_cond}.

To every set $X\subset\mathbb R^n$ one may associate a one-parameter family of open sets by defining for any $\varepsilon>0$
\begin{equation}\label{omega_neighb}
X_{[\varepsilon\omega]}:=\bigcup\limits_{\xi_0\in X}\left\{\xi\in\mathbb R^n\,:\,\,\omega(\xi-\xi_0)<\varepsilon\omega(\xi_0)\right\}\,.
\end{equation}
We call $X_{[\varepsilon\omega]}$ the {\it $[\omega]-$neighborhood} of $X$ of size $\varepsilon$.

\begin{remark}\label{rmk:intorni}
{\rm Since $\left\{\xi\in\mathbb R^n\,;\,\,\omega(\xi-\xi_0)<\varepsilon\omega(\xi_0)\right\}=\emptyset$ when $\omega(\xi_0)\le c/\varepsilon$, where $c$ is the constant in \eqref{bb_cond}, we effectively have 
\begin{equation*}
X_{[\varepsilon\omega]}=\!\!\!\!\bigcup\limits_{\xi_0\in X\,:\,\,\omega(\xi_0)>\frac{c}{\varepsilon}}\!\!\!\left\{\omega(\xi-\xi_0)<\varepsilon\omega(\xi_0)\right\}\,,
\end{equation*}
and for $X$ bounded a constant $\varepsilon_0=\varepsilon_0(X)>0$ exists such that $X_{[\varepsilon\omega]}=\emptyset$ when $0<\varepsilon<\varepsilon_0$.}
\end{remark}

As a consequence of ($\mathcal{SA}$), ($\mathcal{SH}$) and \eqref{strong_sv}, the $[\omega]$-neighborhoods of a set $X$ fulfil the following lemma.
\begin{lemma}\label{mcl_lemma1}
Given $\varepsilon>0$, there exists $0<\varepsilon^\prime<\varepsilon$ such that for every $X\subset\mathbb R^n$
\begin{eqnarray}
\left(X_{[\varepsilon^\prime\omega]}\right)_{[\varepsilon^\prime\omega]}\subset X_{[\varepsilon\omega]}\,;\label{inc_1}\\
\left(\mathbb R^n\setminus X_{[\varepsilon\omega]}\right)_{[\varepsilon^\prime\omega]}\subset\mathbb R^n\setminus X_{[\varepsilon^\prime\omega]}\,.\label{inc_2}
\end{eqnarray}
Moreover there exist constants $\widehat c>0$ and $0<\widehat\varepsilon<1$ such that for all $X\subset\mathbb R^n$ and $0<\varepsilon\le\widehat\varepsilon$
\begin{equation}\label{mcl_impl}
\xi\in X_{[\varepsilon\omega]}\quad\mbox{yields}\quad \omega(\xi)>\frac{\widehat c}{\varepsilon}\,.
\end{equation}
\end{lemma}

\begin{proof}
\eqref{inc_1} and \eqref{inc_2} are direct consequences of ($\mathcal{SA}$), ($\mathcal{SH}$) and \eqref{strong_sv}, see \cite{G} for details.

If $\xi\in X_{[\varepsilon\omega]}$ then $\xi_0\in X$ exists such that
\begin{equation}\label{mcl_eqt1}
\omega(\xi-\xi_0)<\varepsilon\omega(\xi_0)\,,
\end{equation}
hence $\omega(\xi-\xi_0)\ge c$ implies $\omega(\xi_0)>\displaystyle\frac{c}{\varepsilon}$, cf. \eqref{bb_cond} and Remark \ref{rmk:intorni}.

In view of \eqref{strong_sv}
\begin{equation*}
\omega(\xi)=\omega(\xi_0+(\xi-\xi_0))\ge\frac1{C}\omega(\xi_0)>\frac{c}{C\varepsilon}
\end{equation*}
follows from \eqref{mcl_eqt1}, provided that $\varepsilon\le\displaystyle\frac1{C}$ (where $C$ is the same constant involved in \eqref{strong_sv}).
\end{proof}
We use the notion of $[\omega]-$neighborhood of a set to define a microlocal version of the weighted Fourier Lebesgue spaces.
\begin{definition}\label{micro_FL_space}
We say that a distribution $u\in\mathcal S^\prime(\mathbb R^n)$ belongs microlocally to $\mathcal F L^p_{\omega}$ at $X\subset\mathbb R^n$, writing $u\in\mathcal FL^p_{\omega,{\rm mcl}}(X)$, $p\in[1,+\infty]$, if there exists $\varepsilon>0$ such that
\begin{equation}\label{mcl_FLp_1}
\vert u\vert_{X_{[\varepsilon\omega]}}^p:=\int_{X_{[\varepsilon\omega]}}\omega(\xi)^p\vert\widehat{u}(\xi)\vert^p\,d\xi<+\infty
\end{equation}
(with obvious modification for $p=+\infty$).

For $\Omega$ open subset of $\mathbb R^n$, $x_0\in\Omega$ and $X\subset\mathbb R^n$, we say that a distribution $u\in\mathcal D^\prime(\Omega)$ belongs microlocally to $\mathcal F L^p_{\omega}$ on the set $X$ at the point $x_0$, writing $u\in\mathcal F L^p_{\omega, {\rm mcl}}(x_0\times X)$, if there exists a function $\phi\in C^\infty_0(\Omega)$ such that $\phi(x_0)\neq 0$ and $\phi u\in FL^p_{\omega,{\rm mcl}}(X)$.
\end{definition}

\begin{remark}\label{rmk:13}
{\rm 

In view of Remark \ref{rmk:intorni}, condition \eqref{mcl_FLp_1} is meaningful only for {\it unbounded} $X$.}
\end{remark}


We can say that $u\in\mathcal FL^p_{\omega,{\rm mcl}}(X)$ and $u\in\mathcal F L^p_{\omega, {\rm mcl}}(x_0\times X)$ if respectively
\begin{equation}\label{mcl_FLp_2}
\chi_{[\varepsilon\omega]}(\xi)\omega(\xi)\widehat{u}(\xi)\in L^p(\mathbb R^n)
\end{equation}
and
\begin{equation}\label{mcl_FLp_3}
\chi_{[\varepsilon\omega]}(\xi)\omega(\xi)\widehat{\phi u}(\xi)\in L^p(\mathbb R^n)\,,
\end{equation}
where $\chi_{[\varepsilon\omega]}=\chi_{[\varepsilon\omega]}(\xi)$ denotes the characteristic function of the set $X_{[\varepsilon\omega]}$ and $\varepsilon>0$, $\phi=\phi(x)$ are given as in Definition \ref{micro_FL_space}.

\vspace{.5cm}
According to Definition \ref{micro_FL_space} one can introduce the notion of {\it filter of Fourier Lebesgue singularities}, which is in some way the extension of the wave front set of Fourier Lebesgue singularities when we lack the homogeneity properties necessary to use effectively conic neighborhoods.

\begin{definition}\label{FL_sing}
For $u\in\mathcal D^\prime(\Omega)$, $x_0\in\Omega$, $p\in[1,+\infty]$, we call filter of $\mathcal FL^p_\omega-$singularities of $u$ at the point $x_0$ the class of all sets $X\subset\mathbb R^n$ such that $u\in\mathcal FL^p_{\omega, {\rm mcl}}(x_0\times(\mathbb R^n\setminus X))$. It may be easily verified that
\begin{equation}\label{filtroFLp}
\Xi_{\mathcal F L^p_{\omega}\,,\,x_0}u:=\bigcup\limits_{\phi\in C^\infty_0(\Omega)\,,\,\,\phi(x_0)\neq 0}\Xi_{\mathcal F L^p_{\omega}}\phi u\,,
\end{equation}
where for every $v\in\mathcal S^\prime(\mathbb R^n)$, $\Xi_{\mathcal F L^p_{\omega}}v$ is the class  of all sets $X\subset\mathbb R^n$ such that $v\in\mathcal FL^p_{\omega,{\rm mcl}}(\mathbb R^n\setminus X)$.
\end{definition}

$\Xi_{\mathcal F L^p_{\omega}}v$ and $\Xi_{\mathcal F L^p_{\omega}\,,\,x_0} u$ defined above are {\it $[\omega]-$filters}, in the sense that they satisfy the standard filter properties and moreover for all $X\in \Xi_{\mathcal F L^p_{\omega}}v$ (respectively $X\in \Xi_{\mathcal F L^p_{\omega}\,,\,x_0} u$) there exists $\varepsilon>0$ such that $\mathbb R^n\setminus(\mathbb R^n\setminus X)_{[\varepsilon\omega]}\in \Xi_{\mathcal F L^p_{\omega}}v$ (respectively $\mathbb R^n\setminus(\mathbb R^n\setminus X)_{[\varepsilon\omega]}\in \Xi_{\mathcal F L^p_{\omega}\,,\,x_0} u$), see e.g. \cite{TR} for the definition and properties of a filter.

\subsection{Symbols with microlocal regularity in spaces of Fourier Lebesgue type}\label{mcl_FLp_symb_sect}
Throughout the whole section, we assume that $\lambda=\lambda(\xi)$ and $\Lambda=\Lambda(\xi)$ are two continuous weight functions, such that $\lambda$ satisfies condition \eqref{bb_cond} and $\Lambda$ conditions ($\mathcal{SA}$) and ($\mathcal{SH}$).

For given $p\in[1,+\infty]$ and $X\subset\mathbb R^n$, the space $\mathcal F L^p_{\lambda}(\mathbb R^n)\cap\mathcal F L^p_{\Lambda, {\rm mcl}}(X)$ is provided with the {\it inductive limit} locally convex topology defined on it by the family of subspaces
\begin{equation*}
\mathcal F L^p_{\lambda}(\mathbb R^n)\cap\mathcal F L^p_{\Lambda, \varepsilon}(X):=\{u\in \mathcal F L^p_{\lambda}(\mathbb R^n)\,:\,\,\vert u\vert_{X_{[\varepsilon\Lambda]}}<+\infty\}
\end{equation*}
(cf. \eqref{mcl_FLp_1}), endowed with their natural semi-norm
\begin{equation*}
\Vert u\Vert_{\mathcal F L^p_{\lambda}}+\vert u\vert_{X_{[\varepsilon\Lambda]}}\,,\quad\varepsilon>0\,.
\end{equation*}

Analogously for every $x_0\in\Omega$, the space $\mathcal F L^p_{\lambda\,,{\rm loc}}(x_0)\cap\mathcal F L^p_{\Lambda\,, {\rm mcl}}(x_0\times X)$ is provided with the inductive limit topology defined by the subspaces
\begin{equation*}
\mathcal F L^p_{\lambda\,,\phi}\cap\mathcal F L^p_{\Lambda, \varepsilon}(X):=\{u\in\mathcal D^\prime(\Omega)\,:\,\,\phi u\in \mathcal F L^p_{\lambda}(\mathbb R^n)\cap\mathcal F L^p_{\Lambda, \varepsilon}(X)\}\,,
\end{equation*}
endowed with the natural semi-norms
\begin{equation*}
\Vert \phi u\Vert_{\mathcal F L^p_{\lambda}}+\vert \phi u\vert_{X_{[\varepsilon\Lambda]}}\,,\quad \phi\in C^\infty_0(\Omega),\,\,\phi(x_0)\neq 0\,,\,\,\varepsilon>0\,.
\end{equation*}

From the general properties of the inductive limit topology (see e.g. \cite{TR}), it follows that a sequence $\{u_\nu\}$ converges to $u$ in $\mathcal F L^p_{\lambda}(\mathbb R^n)\cap\mathcal F L^p_{\Lambda\,, {\rm mcl}}(X)$ (resp. $\mathcal F L^p_{\lambda\,,{\rm loc}}(x_0)\cap\mathcal F L^p_{\Lambda\,, {\rm mcl}}(x_0\times X)$) if and only if there exists some $\varepsilon>0$ such that
\begin{equation*}
\Vert u_\nu-u\Vert_{\mathcal F L^p_{\lambda}}\to 0\,\,\,\mbox{and}\,\,\,\vert u_\nu-u\vert_{X_{[\varepsilon\Lambda]}}\to 0\,,\,\,\mbox{as}\,\,\nu\to +\infty
\end{equation*}
(resp. there exist $\phi\in C^\infty_0(\Omega)$, with $\phi(x_0)\neq 0$, and $\varepsilon>0$ such that
\begin{equation*}
\Vert \phi(u_\nu-u)\Vert_{\mathcal F L^p_{\lambda}}\to 0\,\,\,\mbox{and}\,\,\,\vert \phi(u_\nu-u)\vert_{X_{[\varepsilon\Lambda]}}\to 0\,,\,\,\mbox{as}\,\,\nu\to +\infty).
\end{equation*}

\begin{definition}\label{def_mcl_symb}
Let $\lambda=\lambda(\xi)$, $\Lambda=\Lambda(\xi)$ be two weight functions as above and $\gamma=\gamma(\xi)$ a further continuous weight function, $x_0\in\Omega$, $X\subset\mathbb R^n$ and $p\in[1,+\infty]$. We say that a distribution $a(x,\xi)\in\mathcal D^\prime(\Omega\times\mathbb R^n)$ belongs to $\mathcal F L^p_{\lambda,\, \Lambda}S_\gamma(x_0\times X)$ if the function $\xi\mapsto a(\cdot,\xi)$ takes values in the space $\mathcal F L^p_{\lambda,{\rm loc}}(x_0)\cap\mathcal F L^p_{\Lambda, {\rm mcl}}(x_0\times X)$ and for some $\phi\in C^\infty_0(\Omega)$ such that $\phi(x_0)\neq 0$ and $\varepsilon>0$ there holds
\begin{equation}\label{symb}
\begin{split}
\Vert a\Vert_{\phi,\lambda,\gamma}&:=\sup\limits_{\xi\in\mathbb R^n}\left\Vert\frac{\lambda(\cdot)\widehat{\phi a} (\cdot,\xi)}{\gamma(\xi)}\right\Vert_{L^p}<+\infty\quad\mbox{and}\\
&\vert a\vert_{\phi,\Lambda,\gamma,\varepsilon,X}:=\sup\limits_{\xi\in\mathbb R^n}\left\Vert\frac{\Lambda(\cdot)\chi_{\varepsilon,\Lambda}(\cdot)\widehat{\phi a} (\cdot,\xi)}{\gamma(\xi)}\right\Vert_{L^p}<+\infty\,,
\end{split}
\end{equation}
where $\widehat{\phi a}(\eta,\xi):=\mathcal F_{x\to\eta}\left(\phi(x)a(x,\xi)\right)(\eta)$ denotes the partial Fourier transform of $\phi(x)a(x,\xi)$ with respect to $x$.
\end{definition}

\begin{theorem}\label{continuity2}
For $p\in[1,+\infty]$, $x_0\in\Omega$, $X\subset\mathbb R^n$, let $\lambda=\lambda(\xi)$, $\Lambda=\Lambda(\xi)$, $\gamma=\gamma(\xi)$, $\sigma=\sigma(\xi)$ be weight functions such that $\lambda$ obeys condition \eqref{B_q}, where $q$ is the conjugate exponent of $p$, $\Lambda$ conditions ($\mathcal{SA}$), ($\mathcal{SH}$), $1/\sigma\in L^q(\mathbb R^n)$ and
\begin{equation}\label{cndt:2}
\quad\sigma(\xi)\preceq\lambda(\xi)\preceq\Lambda(\xi)\preceq\frac{\lambda(\xi)^2}{\sigma(\xi)}\,.
\end{equation}
\begin{itemize}
\item[(i)] If $a(x,\xi)\in \mathcal F L^p_{\lambda,\, \Lambda}S_\gamma(x_0\times X)$ then the corresponding pseudodifferential operator $a(x,D)$ extends to a bounded linear operator
\begin{equation}\label{mcl_cont1}
\mathcal F L^p_{\lambda\gamma}(\mathbb R^n)\cap\mathcal F L^p_{\Lambda\gamma, {\rm mcl}}(X)\rightarrow \mathcal F L^p_{\lambda,{\rm loc}}(x_0)\cap\mathcal F L^p_{\Lambda, {\rm mcl}}(x_0\times X)\,.
\end{equation}
\item[(ii)] If in addition $a(x,D)$ is properly supported, then it extends to a bounded linear operator
\begin{equation}\label{mcl_cont2}
\mathcal F L^p_{\lambda\gamma,{\rm loc}}(x_0)\cap\mathcal F L^p_{\Lambda\gamma, {\rm mcl}}(x_0\times X)\rightarrow \mathcal F L^p_{\lambda,{\rm loc}}(x_0)\cap\mathcal F L^p_{\Lambda, {\rm mcl}}(x_0\times X)\,.
\end{equation}
\end{itemize}
\end{theorem}
\begin{proof}
The statement $(ii)$ follows at once from $(i)$ in view of the definition of a properly supported operator. Thus, let us focus on the proof of $(i)$.

In view of Definition \ref{def_mcl_symb}, there exist $\varepsilon>0$ and $\phi\in C_0^\infty(\Omega)$, with $\phi(x_0)\neq 0$, such that conditions in \eqref{symb} are satisfied.
We are going first to prove that
\begin{equation}\label{reg_1}
\phi(x)a(x,D)u\in\mathcal F L^p_{\lambda}(\mathbb R^n)\,,
\end{equation}
as long as $u\in\mathcal F L^p_{\lambda\gamma}(\mathbb R^n)$. Let us denote for short
\begin{equation*}
a_{\phi}(x,\xi):=\phi(x)a(x,\xi)\,.
\end{equation*}
In order to check \eqref{reg_1} it is enough to apply the result of Proposition \ref{continuity1} to the symbol $a_{\phi}(x,\xi)\in\mathcal F L^p_{\lambda}S_\gamma$ (cf. Definition \ref{symbols}) where, restoring the notations used there, we set
\begin{equation*}
\omega_1(\zeta)=\lambda(\zeta)\gamma(\zeta)\,,\quad\omega(\zeta)=\omega_2(\zeta)=\lambda(\zeta)\,.
\end{equation*}
Under the previous positions, the condition \eqref{B_q1} of Proposition \ref{continuity1} reduces to require that $\lambda=\lambda(\zeta)$ satisfies \eqref{B_q}. From Proposition \ref{continuity1} we also deduce the continuity of $a(x,D)$ as a linear map from $\mathcal F L^p_{\lambda\gamma}(\mathbb R^n)$ into $\mathcal F L^p_{\lambda,{\rm loc}}(x_0)$.

It remains to show that
\begin{equation}\label{reg_2}
a_\phi(x,D)u\in\mathcal F L^p_{\Lambda,{\rm mcl}}(X)\,,
\end{equation}
when $u\in\mathcal F L^p_{\lambda\gamma}(\mathbb R^n)\cap\mathcal F L^p_{\Lambda\gamma,{\rm mcl}}(X)$, as well as the continuity of $a(x,D)$ as an operator acting on the aforementioned spaces. Throughout the rest of the proof, we will denote by $C$ some positive constant that is independent of the symbol $a(x,\xi)$ and the function $u(x)$ and may possibly differ from an occurrence to another.

In view of Lemma \ref{mcl_lemma1}, there exists some $0<\varepsilon^\prime<\varepsilon$ such that
\begin{equation*}
\left(\mathbb R^n\setminus X_{[\varepsilon\Lambda]}\right)_{[\varepsilon^\prime\Lambda]}\subset\mathbb R^n\setminus X_{[\varepsilon^\prime\Lambda]}\,.
\end{equation*}
Let us denote for short
\begin{equation}\label{char}
\chi(\zeta):=\chi_{[\varepsilon^\prime\Lambda]}(\zeta)\,,\quad \chi_1(\zeta):=\chi_{[\varepsilon\Lambda]}(\zeta)\,,\quad \chi_2(\zeta):=1-\chi_{[\varepsilon\Lambda]}(\zeta)
\end{equation}
and write
\begin{equation*}
\widehat{a_\phi}(\zeta,\xi)\widehat u(\xi)=\sum\limits_{i,j=1,2}\chi_i(\zeta)\widehat{a_\phi}(\zeta,\xi)\chi_j(\xi)\widehat{u}(\xi)\,.
\end{equation*}
Then in view of \eqref{representation1} and condition ($\mathcal{SA}$) for $\Lambda$, we find
\begin{equation}\label{stima_cont_2_1}
\begin{split}
&\vert\chi(\eta)\Lambda(\eta)\widehat{a_\phi(\cdot,D)u}(\eta)\vert\\
& \le C(2\pi)^{-n}\int\chi(\eta)\left\{\Lambda(\eta-\xi)+\Lambda(\xi)\right\}\vert\widehat{a_\phi}(\eta-\xi,\xi)\vert\,\vert\widehat{u}(\xi)\vert d\xi\\
& \le C(2\pi)^{-n}\int\chi(\eta)\sum\limits_{i,j=1,2}\chi_i(\eta-\xi)\Lambda(\eta-\xi)\vert\widehat{a_\phi}(\eta-\xi,\xi)\vert\chi_j(\xi)\vert\widehat{u}(\xi)\vert d\xi\\
&+C(2\pi)^{-n}\int\chi(\eta)\sum\limits_{i,j=1,2}\chi_i(\eta-\xi)\vert\widehat{a_\phi}(\eta-\xi,\xi)\vert\chi_j(\xi)\Lambda(\xi)\vert\widehat{u}(\xi)\vert d\xi\\
&=\mathcal I_1u(\eta)+\mathcal I_2u(\eta)\,.
\end{split}
\end{equation}
Let us set
\begin{equation}\label{positions1}
\begin{split}
& g_1(\zeta,\xi)=\chi_1(\zeta)\Lambda(\zeta)\gamma(\xi)^{-1}\vert\widehat{a_\phi}(\zeta,\xi)\vert\,;\\
& g_2(\zeta,\xi)=\chi_2(\zeta)\sigma(\zeta)\Lambda(\zeta)\gamma(\xi)^{-1}\Lambda(\xi)^{-1}\vert\widehat{a_\phi}(\zeta,\xi)\vert\,;\\
& \tilde{g}_2(\zeta,\xi)=\chi_2(\zeta)\sigma(\zeta)^{1/2}\Lambda(\zeta)^{1/2}\gamma(\xi)^{-1}\vert\widehat{a_\phi}(\zeta,\xi)\vert\,;\\
& v_1(\xi)=\chi_1(\xi)\gamma(\xi)\sigma(\xi)\vert\widehat{u}(\xi)\vert\,;\quad \tilde v_1(\xi)=\chi_1(\xi)\gamma(\xi)\Lambda(\xi)\vert\widehat{u}(\xi)\vert\,;\\
& v_2(\xi)=\chi_2(\xi)\gamma(\xi)\sigma(\xi)\vert\widehat{u}(\xi)\vert\,;\\
& \tilde v_2(\zeta,\xi)=\chi_2(\xi)\sigma(\xi)^{1/2}\Lambda(\zeta)^{1/2}\gamma(\xi)\vert\widehat{u}(\xi)\vert\,.
\end{split}
\end{equation}
Then the first integral in the right-hand side of \eqref{stima_cont_2_1} can be rewritten as
\begin{equation}\label{decomp1}
\begin{split}
&\mathcal I_1u(\eta)=\int\chi(\eta)\frac1{\sigma(\xi)}g_1(\eta-\xi,\xi)v_1(\xi)d\xi\\
&+\!\!\int\chi(\eta)\frac1{\sigma(\xi)}g_1(\eta-\xi,\xi)v_2(\xi)d\xi+\int\chi(\eta)\frac1{\sigma(\eta-\xi)}g_2(\eta-\xi,\xi)\tilde v_1(\xi)d\xi\\
&+\!\!\int\chi(\eta)\frac1{\sqrt{\sigma(\xi)\sigma(\eta-\xi)}}\tilde g_2(\eta-\xi,\xi)\tilde v_2(\eta-\xi,\xi)d\xi\,.
\end{split}
\end{equation}
In view of the assumptions in \eqref{cndt:2} it is easy to see that all the above functions $v_1$, $v_2$, $\tilde v_1$ defined in \eqref{positions1} belong to $L^p(\mathbb R^n)$, if  $u\in\mathcal F L^p_{\lambda\gamma}(\mathbb R^n)\cap\mathcal F L^p_{\Lambda\gamma,{\rm mcl}}(X)$, with the following estimates
\begin{equation}\label{stima_cont_2_2}
\Vert v_1\Vert_{L^p}\le\vert u\vert_{X_{[\varepsilon\Lambda\gamma]}}\,,\quad \Vert \tilde v_1\Vert_{L^p}\le\vert u\vert_{X_{[\varepsilon\Lambda\gamma]}}\,,\quad\Vert v_2\Vert_{L^p}\le\Vert u\Vert_{\mathcal F L^p_{\lambda\gamma}}\,.
\end{equation}
Moreover the functions
\begin{equation}\label{Fi}
F_1(\eta,\xi):=\frac{\chi(\eta)}{\sigma(\xi)}\,,\,\,\,F_2(\eta,\xi):=\frac{\chi(\eta)}{\sigma(\eta-\xi)}\,,\,\,\, F_3(\eta,\xi):=\frac{\chi(\eta)}{\sqrt{\sigma(\xi)\sigma(\eta-\xi)}}
\end{equation}
belong to the space $\mathcal L_2^{\infty, q}(\mathbb R^{2n})$, with the estimates
\begin{equation}\label{est_Fi}
\Vert F_i\Vert_{\mathcal L_2^{\infty,q}}\le\Vert 1/\sigma\Vert_{L^q}\,,\quad i=1,2,3\,.
\end{equation}

Again from \eqref{cndt:2} and \eqref{symb} we easily obtain that $g_1(\zeta,\xi)\in\mathcal L^{p,\infty}_1(\mathbb R^{2n})$ and satisfies the estimate
\begin{equation*}
\Vert g_1\Vert_{\mathcal L^{p,\infty}_1}\le\vert a\vert_{\phi,\Lambda,\gamma,\varepsilon,X}\,.
\end{equation*}
In view of the previous analysis, the first two integral operators involved in \eqref{decomp1} have the form of the operator considered in Lemma \ref{lemma_cont2}. Thus from Lemma \ref{lemma_cont2} and the estimates collected above we get
\begin{equation}\label{cont_1}
\begin{split}
\left\Vert\int\chi(\cdot)\right.&\left.\!\!\frac1{\sigma(\xi)}g_1(\cdot-\xi,\xi)v_1(\xi)d\xi\right\Vert_{L^p}\!\!\!+\left\Vert\int\chi(\cdot)\frac1{\sigma(\xi)}g_1(\cdot-\xi,\xi)v_2(\xi)d\xi\right\Vert_{L^p}\\
&\le \Vert 1/\sigma\Vert_{L^q}\vert a\vert_{\phi,\Lambda,\gamma,\varepsilon,X}\left\{\vert u\vert_{X_{[\varepsilon\Lambda\gamma]}}+\Vert u\Vert_{\mathcal F L^p_{\lambda\gamma}}\right\}\,.
\end{split}
\end{equation}

Concerning the third integral operator in \eqref{decomp1}, we notice that the involved function $g_2(\zeta,\xi)$ vanishes when $\zeta+\xi\notin X_{[\varepsilon^\prime\Lambda]}$, due to the presence of the characteristic function $\chi$. For $\zeta+\xi\in X_{[\varepsilon^\prime\Lambda]}$ and $\zeta\in\mathbb R^n\setminus X_{[\varepsilon\Lambda]}$ it follows that $\Lambda(\zeta)\le\frac1{\varepsilon^\prime}\Lambda(\xi)$; indeed the converse inequality $\Lambda((\zeta+\xi)-\zeta)=\Lambda(\xi)<\varepsilon^\prime\Lambda(\zeta)$ should mean that $\zeta+\xi\in (\mathbb R^n\setminus X_{[\varepsilon\Lambda]})_{[\varepsilon^\prime\Lambda]}\subset\mathbb R^n\setminus X_{[\varepsilon^\prime\Lambda]}$. Hence we get
\begin{equation}\label{est_g2}
\vert g_2(\zeta,\xi)\vert\le\frac1{\varepsilon^\prime}\chi_2(\zeta)\sigma(\zeta)\gamma(\xi)^{-1}\vert\widehat{a_\phi}(\zeta,\xi)\vert\,,
\end{equation}
and, using also $\sigma\preceq\lambda$,
\begin{equation*}
\Vert g_2(\cdot,\xi)\Vert_{L^p}\le \frac1{\varepsilon^\prime}\Vert\chi_2(\cdot)\sigma(\cdot)\gamma(\xi)^{-1}\vert\widehat{a_\phi}(\cdot,\xi)\Vert_{L^p}\le \frac{C}{\varepsilon^\prime}\Vert a\Vert_{\phi,\lambda,\gamma}\,.
\end{equation*}
This yields that $g_2(\zeta,\xi)\in\mathcal L_1^{p,\infty}(\mathbb R^{2n})$ with norm bounded by
\begin{equation*}
\Vert g_2\Vert_{\mathcal L_1^{p,\infty}}\le\frac{C}{\varepsilon^\prime}\Vert a\Vert_{\phi,\lambda,\gamma}\,.
\end{equation*}
Hence we may apply again Lemma \ref{lemma_cont2} to the third operator in \eqref{decomp1}, and using also the estimates \eqref{est_Fi}, \eqref{stima_cont_2_2} we find
\begin{equation}\label{cont_2}
\left\Vert\int\chi(\cdot)\frac1{\sigma(\cdot-\xi)}g_2(\cdot-\xi,\xi)\tilde v_1(\xi)d\xi\right\Vert_{L^p}\le \frac{C}{\varepsilon^\prime}\Vert 1/\sigma\Vert_{L^q}\Vert a\Vert_{\phi,\lambda,\gamma}\vert u\vert_{X_{[\varepsilon\Lambda\gamma]}}\,.
\end{equation}
Let us consider now the fourth integral operator in \eqref{decomp1}. Applying the same argument used to provide the estimate \eqref{est_g2}, we obtain
\begin{equation}\label{est_tildev2}
\vert \tilde v_2(\zeta,\xi)\vert\le\frac1{\varepsilon^\prime}\chi_2(\xi)\sigma(\xi)^{1/2}\Lambda(\xi)^{1/2}\gamma(\xi)\vert\widehat{u}(\xi)\vert\,.
\end{equation}
Thanks to \eqref{cndt:2}, $\sigma^{1/2}\Lambda^{1/2}\preceq\lambda$, then
\begin{equation}\label{est_int_oprt}
\begin{split}
\left\vert\int \right.&\left.\frac{\chi(\eta)}{\sqrt{\sigma(\xi)\sigma(\eta-\xi)}}\tilde g_2(\eta-\xi,\xi)\tilde v_2(\eta-\xi,\xi)d\xi\right\vert\\
&\le \frac{C}{\varepsilon^\prime}\int\frac{\chi(\eta)}{\sqrt{\sigma(\xi)\sigma(\eta-\xi)}}\vert\tilde g_2(\eta-\xi,\xi)\vert\lambda(\xi)\gamma(\xi)\vert\widehat{u}(\xi)\vert\,d\xi\,.
\end{split}
\end{equation}
On the other hand, using again $\sigma^{1/2}\Lambda^{1/2}\preceq\lambda$ and $a_\phi(\cdot,\xi)/\gamma(\xi)\in\mathcal F L^p_{\lambda}(\mathbb R^n)$, uniformly with respect to $\xi$, we establish that $\tilde{g}_2(\zeta,\xi)$ belongs to $\mathcal L_1^{p,\infty}(\mathbb R^{2n})$ and satisfies the estimate
\begin{equation}\label{est_tildeg2}
\Vert\tilde g_2\Vert_{\mathcal L_1^{p,\infty}}\le\Vert\sigma(\cdot)^{1/2}\Lambda(\cdot)^{1/2}\gamma(\xi)^{-1}\widehat{a_\phi}(\cdot,\xi)\Vert_{L^p}\le C\Vert a\Vert_{\phi,\lambda,\gamma}\,.
\end{equation}
Since $\lambda(\xi)\gamma(\xi)\vert\widehat{u}(\xi)\vert\in L^p(\mathbb R^n)$ (as $u\in\mathcal FL^p_{\lambda\gamma}(\mathbb R^n)$) and $F_3(\eta,\xi)=\frac{\chi(\eta)}{\sqrt{\sigma(\xi)\sigma(\eta-\xi)}}$ belongs to $\mathcal L_2^{\infty,q}(\mathbb R^{2n})$, the integral operator in the right-hand side of \eqref{est_int_oprt} satisfies the assumptions of Lemma \ref{lemma_cont2}, then from \eqref{est_Fi} and \eqref{est_g2} we find
\begin{equation}\label{cont_3}
\begin{split}
\left\Vert\int \right.&\left.\frac{\chi(\cdot)}{\sqrt{\sigma(\xi)\sigma(\cdot-\xi)}}\tilde g_2(\cdot-\xi,\xi)\tilde v_2(\cdot-\xi,\xi)d\xi\right\Vert_{L^p}\\
&\le \frac{C}{\varepsilon^\prime}\Vert 1/\sigma\Vert_{L^q}\Vert a\Vert_{\phi,\lambda,\gamma}\Vert u\Vert_{\mathcal FL^p_{\lambda\gamma}}\,.
\end{split}
\end{equation}
Summing up the estimates \eqref{cont_1}, \eqref{cont_2}, \eqref{cont_3} the $L^p-$norm of $\mathcal I_1u$ in the right-hand side of \eqref{stima_cont_2_1} is estimated by
\begin{equation}\label{cont_4}
\Vert\mathcal I_1u\Vert_{L^p}\le \frac{C}{\varepsilon^\prime}\Vert 1/\sigma\Vert_{L^q}\left(\vert a\vert_{\phi,\Lambda,\gamma,\varepsilon,X}+\Vert a\Vert_{\phi,\lambda,\gamma}\right)\left(\vert u\vert_{X_{[\varepsilon\Lambda\gamma]}}+\Vert u\Vert_{\mathcal F L^p_{\lambda\gamma}}\right)\,.
\end{equation}
The second integral $\mathcal I_2u(\eta)$ in \eqref{stima_cont_2_1} can be handled similarly as before to provide for its $L^p-$norm the same bound as in \eqref{cont_4}. From \eqref{stima_cont_2_1} we then get
\begin{equation}\label{cont_5}
\begin{split}
\Vert\chi&\Lambda\widehat{a_\phi(\cdot,D)u}\Vert_{L^p}\\
&\le \frac{C}{\varepsilon^\prime}\Vert 1/\sigma\Vert_{L^q}\left(\vert a\vert_{\phi,\Lambda,\gamma,\varepsilon,X}+\Vert a\Vert_{\phi,\lambda,\gamma}\right)\left(\vert u\vert_{X_{[\varepsilon\Lambda\gamma]}}+\Vert u\Vert_{\mathcal F L^p_{(\lambda\gamma)}}\right)
\end{split}
\end{equation}
which proves \eqref{reg_2} and shows the continuity of $a(x,D)$ as a linear map from $\mathcal F L^p_{\lambda\gamma}(\mathbb R^n)\cap\mathcal F L^p_{\Lambda\gamma, {\rm mcl}}(X)$ into $\mathcal F L^p_{\lambda,{\rm loc}}(x_0)\cap\mathcal F L^p_{\Lambda, {\rm mcl}}(x_0\times X)$.
\end{proof}
\begin{remark}\label{rmk:14}
{\rm
Let the same hypotheses of Theorem \ref{continuity2} be satisfied.
Clearly every $v=v(x)\in\mathcal FL^p_{\lambda,{\rm loc}}(x_0)\cap\mathcal FL^p_{\Lambda,{\rm mcl}}(x_0\times X)$ is a $\xi-$independent symbol in the class $\mathcal F L^p_{\lambda,\,\Lambda}S_\gamma(x_0\times X)$ corresponding to the weight function $\gamma(\xi)\equiv 1$, and the product of smooth functions by the multiplier $v$ defines a properly supported ``zeroth order'' operator. Therefore we find that the product of any two elements $u, v\in \mathcal FL^p_{\lambda,{\rm loc}}(x_0)\cap\mathcal FL^p_{\Lambda,{\rm mcl}}(x_0\times X)$ still belongs to the same space (giving a continuous bilinear mapping), as a direct application of Theorem \ref{continuity2}. Similarly as in the proof of Corollary \ref{composition_prop}, see also the subsequent Remark \ref{rmk:3.1}, one can deduce that the composition of a vector-valued distribution $u=(u_1,\dots,u_N)\in\left(\mathcal FL^p_{\lambda,{\rm loc}}(x_0)\cap\mathcal FL^p_{\Lambda,{\rm mcl}}(x_0\times X)\right)^N$ with some nonlinear function $F=F(x,\zeta)$ of $x\in\mathbb R^n$ and $\zeta\in\mathbb C^N$, which is locally smooth with respect to $x$ on some neighborhood of $x_0$ and entire analytic with respect to $\zeta$ in the sense of Remark \ref{rmk:3.1}, is again a distribution in $\mathcal FL^p_{\lambda,{\rm loc}}(x_0)\cap\mathcal FL^p_{\Lambda,{\rm mcl}}(x_0\times X)$.

Let us even point out that in the particular case where $\lambda\equiv\Lambda$ the assumption \eqref{cndt:2} in Theorem \ref{continuity2} reduces to $\sigma\preceq\lambda$. In such a case $\mathcal FL^p_{\lambda,{\rm loc}}(x_0)\cap\mathcal FL^p_{\Lambda,{\rm mcl}}(x_0\times X)\equiv \mathcal FL^p_{\lambda,{\rm loc}}(x_0)$ and $\mathcal FL^p_{\lambda,\Lambda}S_\gamma(x_0\times X)\equiv\mathcal FL^p_{\lambda}S_\gamma(V_{x_0})$ for a suitable neighborhood $V_{x_0}$ of $x_0$, see Definition \ref{symbols}, hence the statement of Theorem \ref{continuity2} reduces to a particular case of the statement of Proposition \ref{continuity1} (where $\omega_1=\gamma\lambda$, $\omega=\omega_2=\lambda$) under slightly more restrictive assumptions; indeed a sub-additive weight function $\lambda$ satisfying $\sigma\preceq\lambda$ for $1/\sigma\in L^q(\mathbb R^n)$ also fulfils condition \eqref{B_q} with the same $q$ (that is the assumption required by Proposition \ref{continuity1}), in view of Proposition \ref{wf_relationship}.{\it ii}.
}
\end{remark}
\section{Propagation of singularities}\label{appl_sct}
In this Section, we give some applications to the local and microlocal regularity of semilinear partial(pseudo)differential equations in weighted Fourier Lebesgue spaces. 

The smooth symbols we consider in this Section are related to a suitable subclass of the weight functions introduced in Section \ref{wf_sct}. More precisely, we consider a continuous function $\lambda:\mathbb R^n\rightarrow]0,+\infty[$ satisfying the following:
\begin{equation}\label{p_g_b}
\lambda(\xi)\ge \frac1{C}(1+\vert\xi\vert)^\nu\,,\quad\forall\,\xi\in\mathbb R^n\,;
\end{equation}
\begin{equation}\label{s_v}
\frac1{C}\le\frac{\lambda(\xi)}{\lambda(\eta)}\le C\,,\quad\mbox{as long as}\,\,\vert\xi-\eta\vert\le \frac1{C}\lambda(\eta)^{1/\mu}\,,
\end{equation}
for suitable constants $C\ge 1$, $0<\nu\le\mu$.

Thanks to Proposition \ref{wf_relationship}, it is clear that $\lambda(\xi)$ is a weight function; indeed it also satisfies the temperance condition ($\mathcal T$) for $N=\mu$. 

All the weight functions described in the examples 1--3 given in Section \ref{wf_sct} obey the assumptions \eqref{p_g_b}, \eqref{s_v}.

\smallskip
For $r\in\mathbb R$, $\rho\in]0,1/\mu]$, we define $S^r_{\rho,\lambda}$ as the class of smooth functions $a(x,\xi)\in C^\infty(\mathbb R^{2n})$ whose derivatives decay according to the following estimates
\begin{equation}\label{smooth_symb}
\vert\partial^\alpha_\xi\partial^\beta_x a(x,\xi)\vert\le C_{\alpha,\beta}\lambda(\xi)^{r-\rho\vert\alpha\vert}\,,\qquad\forall\,(x,\xi)\in\mathbb R^{2n}\,.
\end{equation}
If $\Omega$ is an open subset of $\mathbb R^n$, the local class $S^r_{\rho,\lambda}(\Omega)$ is the set of functions $a(x,\xi)\in C^\infty(\Omega\times\mathbb R^n)$ such that $\phi(x)a(x,\xi)\in S^r_{\rho,\lambda}$ for all $\phi\in C^\infty_0(\Omega)$. We will adopt the shortcut
\begin{equation*}
S^r_{\lambda}:=S^r_{1/\mu,\lambda}\,,\qquad S^r_{\lambda}(\Omega):=S^r_{1/\mu,\lambda}(\Omega)\,.
\end{equation*}
A symbol $a(x,\xi)\in S^r_{\lambda}(\Omega)$ (and the related pseudodifferential operator) is said to be $\lambda-${\it elliptic} if for every compact subset $K$ of $\Omega$ some positive constants $c_K$ and $R_K>1$ exist such that
\begin{equation}\label{ell}
\vert a(x,\xi)\vert\ge c_K\lambda(\xi)^r\,,\qquad\forall\,x\in K\,\,\,\mbox{and}\,\,\,\vert\xi\vert\ge R_K\,.
\end{equation}
Let us also observe that
$\bigcap\limits_{r\in\mathbb R}S^r_{\rho,\lambda}(\Omega)=S^{-\infty}(\Omega)$, where in the classic terms $S^{-\infty}(\Omega)$ is the class of symbols $a(x,\xi)\in C^\infty(\Omega\times\mathbb R^n)$ such that for arbitrarily large $\theta>0$, for all multi-indices $\alpha,\beta\in\mathbb Z^n_+$ and every compact set $K\subset\Omega$ there holds
\begin{equation*}
\vert\partial^\alpha_\xi\partial^\beta_x a(x,\xi)\vert\le C_{\alpha,\beta,\theta}(1+\vert\xi\vert)^{-\theta}\,,\quad\forall\,x\in K\,,\,\,\,\forall\,\xi\in\mathbb R^n\,.
\end{equation*}
Pseudodifferential operators with symbols $a(x,\xi)\in S^{-\infty}(\Omega)$ are {\it regularizing} operators in the sense that they define linear bounded operators $a(x,D):\mathcal E^\prime(\Omega)\rightarrow C^\infty(\Omega)$.

The weighted symbol classes $S^r_{\rho,\lambda}(\Omega)$ considered above are a special case of the more general classes $S_{m,\Lambda}(\Omega)$, associated to the weight function $m(\xi)=\lambda(\xi)^r$ and the weight vector $\Lambda(\xi)=(\lambda(\xi)^\rho,\dots,\lambda(\xi)^\rho)$, as defined and studied in \cite[Definition 1.1]{GM3}. For the weighted symbol classes $S^r_{\rho,\lambda}(\Omega)$, a complete symbolic calculus is available, cf. \cite[Sect.1]{GM3}; in particular, the existence of a {\it parametrix} of any elliptic pseudodifferential operator is guaranteed.
\begin{proposition}\label{parametrix}
Let $a(x.\xi)$ be a $\lambda-$elliptic symbol in $S^r_{\rho,\lambda}(\Omega)$. Then a symbol $b(x,\xi)\in S^{-r}_{\rho, \lambda}(\Omega)$ exists such that the operator $b(x,D)$ is properly supported and satisfies
\begin{equation*}
b(x,D)a(x,D)=I+c(x,D)\,,
\end{equation*}
where $I$ denotes the identity operator and $c(x,D)$ is a regularizing pseudodifferential operator.
\end{proposition}
The following inclusion
\begin{equation}\label{symb_inclusion}
S^r_{\rho,\lambda}(\Omega)\subset\mathcal FL^p_{\omega}S_{\lambda^r}(\Omega)
\end{equation}
holds true, with continuous imbedding, for all $r\in\mathbb R$, $\rho\in]0,1/\mu]$, $p\in[1,+\infty]$ and any weight function $\omega(\xi)$. As a consequence of Proposition \ref{continuity1} we then obtain the following continuity result.
\begin{proposition}\label{continuity3}
Let $\omega(\xi)$ be any weight function and $p\in[1,+\infty]$. Then every pseudodifferential operator with symbol $a(x,\xi)\in S^r_{\rho,\lambda}(\Omega)$ extends to a linear bounded operator
\begin{equation*}
a(x,D):\mathcal FL^p_{\lambda^r\omega}(\mathbb R^n)\rightarrow\mathcal FL^p_{\omega,{\rm loc}}(\Omega)\,.
\end{equation*}
If in addition $a(x,D)$ is properly supported, then the latter extends to a linear bounded operator
\begin{equation*}
a(x,D):\mathcal FL^p_{\lambda^r\omega,{\rm loc}}(\Omega)\rightarrow\mathcal FL^p_{\omega,{\rm loc}}(\Omega)\,.
\end{equation*}
\end{proposition}
\begin{proof}
In view of \eqref{symb_inclusion}, it is enough to observe that for any weight function $\omega(\xi)$, another weight function $\tilde{\omega}(\xi)$ can be found in such a way that
\begin{equation}\label{B_q2}
\sup\limits_{\xi\in\mathbb R^n}\left\Vert\frac{\omega(\xi)}{\omega(\cdot)\tilde{\omega}(\xi-\cdot)}\right\Vert_{L^q}<+\infty\,,
\end{equation}
where $q\in[1,+\infty]$ is the conjugate exponent of $p$; for instance, one can take $\tilde\omega(\xi)=(1+\vert\xi\vert)^{\tilde N}$, with $\tilde N>0$ sufficiently large. Then the result follows at once, by noticing that $a(x,\xi)$ belongs to $\mathcal FL^p_{\tilde\omega}S_{\lambda^r}(\Omega)$ and \eqref{B_q2} is nothing but condition \eqref{B_q1}, where $\gamma$, $\omega_1$, $\omega_2$ and $\omega$ in Proposition \ref{continuity1} are replaced respectively by $\lambda^r$, $\lambda^r\omega$, $\omega$ and $\tilde\omega$.
\end{proof}
\subsection{Local regularity results}\label{loc_reg_sct}
Let $\lambda=\lambda(\xi)$ be a given continuous weigh function satisfying the assumptions \eqref{p_g_b} and \eqref{s_v}. We consider a nonlinear pseudodifferential equation of the following type
\begin{equation}\label{semilin_eqt}
a(x,D)u+ F(x,b_i(x,D)u)_{1\le i\le M}=f(x)\,,
\end{equation}
where $u=u(x)$ is defined on some open set $\Omega\subseteq\mathbb R^n$ and $a(x,D)$ is a properly supported pseudodifferential operator with symbol $a(x,\xi)\in S^r_{\lambda}(\Omega)$ for given $r>0$. $F(x,b^i(x,D)u)_{1\le i\le M}$ stands for a nonlinear function of $x\in\Omega$ and $b
_1(x,D)u$, $b_2(x,D)u$,..., $b_M(x,D)u$ where $b_i(x,D)$ are still properly supported pseudodifferential operators, and $f=f(x)$ is a given forcing term. We require the equation \eqref{semilin_eqt} to be {\it semilinear} by assuming that the operators involved in the nonlinear part $F(x,b^i(x,D)u)$ have order strictly smaller than the order of the linear part $a(x,D)u$, that is
\begin{equation}\label{lower_order}
b^i(x,\xi)\in S^{r-\varepsilon}_{\lambda}(\Omega)\qquad\mbox{for}\,\, i=1,\dots,M\,,
\end{equation}
for suitable $0<\varepsilon<r$.

For $s\in\mathbb R$, $p\in[1,+\infty]$, let us set
\begin{equation*}
\mathcal FL^p_{s,\lambda}(\mathbb R^n):=\mathcal FL^p_{\lambda^s}(\mathbb R^n)\,,\quad\mathcal FL^p_{s,\lambda,{\rm loc}}(\Omega):=\mathcal FL^p_{\lambda^s,{\rm loc}}(\Omega)\,.
\end{equation*}
The following regularity result can be proved.
\begin{proposition}\label{reg_prop1}
Let the symbol $a(x,\xi)\in S^r_{\lambda}(\Omega)$ be $\lambda-$elliptic and the function $F=F(x,\zeta)$ obey the assumptions collected in Remark \ref{rmk:3.1}. For a given $p\in[1,+\infty]$, take a real number $t$ such that $\lambda^{t-r+\varepsilon}$ fulfils condition \eqref{B_q} with $q$ the conjugate exponent of $p$. If $u\in\mathcal FL^p_{t,\lambda,{\rm loc}}(\Omega)$ is any solution of the equation \eqref{semilin_eqt}, with forcing term $f\in\mathcal F L^p_{s-r,\lambda,{\rm loc}}(\Omega)$ for some $s>t$, then $u\in \mathcal FL^p_{s,\lambda,{\rm loc}}(\Omega)$. 

If in particular $u\in\mathcal FL^p_{t,\lambda,{\rm loc}}(\Omega)$ solves the equation \eqref{semilin_eqt} with $f=0$ (that is the equation \eqref{semilin_eqt} is homogeneous) then $u\in C^\infty(\Omega)$.
\end{proposition}
\begin{proof}
Because of Proposition \ref{continuity3} and the assumption \eqref{lower_order}, $b_i(x,D)u\in\mathcal FL^p_{t-r+\varepsilon, \lambda,{\rm loc}}(\Omega)$ for all $i=1,\dots,M$. Since $\lambda^{t-r+\varepsilon}$ satisfies \eqref{B_q}, Corollary \ref{composition_prop} also implies $F(x,b_i(x,D)u)\in\mathcal FL^p_{t-r+\varepsilon,\lambda,{\rm loc}}(\Omega)$ (cf. Remark \ref{rmk:3.1}).

If $t+\varepsilon\ge s$ then $a(x,D)u=-F(x,b_i(x,D)u)+f\in \mathcal F L^p_{s-r,\lambda,{\rm loc}}(\Omega)$ hence $u\in \mathcal F L^p_{s,\lambda,{\rm loc}}(\Omega)$ because of the $\lambda-$ellipticity of $a(x,D)$.

If on the contrary $t+\varepsilon<s$, applying again the $\lambda-$ellipticity of $a(x,D)$, from $a(x,D)u=-F(x,b_i(x,D)u)+f\in \mathcal F L^p_{t-r+\varepsilon,\lambda,{\rm loc}}(\Omega)$ we derive $u\in \mathcal F L^p_{t+\varepsilon,\lambda,{\rm loc}}(\Omega)$. In the latter case, we may repeat the same arguments above, where now $t$ is replaced by $t+\varepsilon$\,\footnote{Let us notice in particular that if the weight function $\lambda^{t-r+\varepsilon}$ satisfies condition \eqref{B_q}, then the same is true for any power of $\lambda$ with exponent greater than $t-r+\varepsilon$, in view of Proposition \ref{algebra_prop4}.}. After that we get $F(x,b_i(x,D)u)\in\mathcal FL^p_{t-r+2\varepsilon,\lambda,{\rm loc}}(\Omega)$ and, provided that $t+2\varepsilon<s$, $u\in \mathcal F L^p_{t+2\varepsilon,\lambda,{\rm loc}}(\Omega)$. It is now clear that the second part of the argument above can be iterated $N$ times, up to get $F(x,b_i(x,D)u)\in\mathcal FL^p_{t-r+N\varepsilon,\lambda,{\rm loc}}(\Omega)$ with $t+N\varepsilon\ge s$; hence $a(x,D)u=-F(x,b_i(x,D)u)+f\in \mathcal F L^p_{s-r,\lambda,{\rm loc}}(\Omega)$ implies $u\in \mathcal F L^p_{s,\lambda,{\rm loc}}(\Omega)$ from the $\lambda-$ellipticity of $a(x,D)$.

The second part of the theorem, concerning the case $f=0$, follows at once from the first one; in this case the argument above can be applied for arbitrarily large $s$, thus $u\in\bigcap\limits_{s\ge t}\mathcal FL^p_{s,\lambda,{\rm loc}}(\Omega)\subset C^\infty(\Omega)$.
\end{proof}
\begin{remark}\label{rmk:100}
{\rm Let us suppose that the weight function $\lambda=\lambda(\xi)$ fulfils condition ($\mathcal{SA}$) (respectively condition ($\mathcal G$)), besides \eqref{p_g_b} and \eqref{s_v}. Then $\lambda^{t-r+\varepsilon}$ satisfies condition \eqref{B_q} if $t>r+\frac{n}{\nu q}-\varepsilon$ (respectively $t>r+\frac{n}{(1-\delta)\nu q}-\varepsilon$) is assumed.}
\end{remark}
\subsection{Microlocal regularity results}\label{mcl_reg_sct}
The results presented in this section apply to a class of weight functions which is smaller than the one considered in Section \ref{loc_reg_sct}. More precisely here we deal with a continuous function $\lambda:\mathbb R^n\rightarrow]0,+\infty[$ which satisfies ($\mathcal{SA}$), ($\mathcal{SH}$) and obeys the following
\begin{itemize}
\item[($\mathcal{PG}$)]{\bf polynomial growth conditions}: for suitable constants $C\ge 1$, $0<\nu\le\mu$.
\begin{equation}\label{p_g}
\frac1{C}(1+\vert\xi\vert)^{\nu}\le\lambda(\xi)\le C(1+\vert\xi\vert)^{\mu}\,,\quad\forall\,\xi\in\mathbb R^n\,.
\end{equation}
\end{itemize}

\begin{remark}\label{rmk:15}
{\rm It is known from the previous section that such a function $\lambda$ also satisfies condition \eqref{strong_sv}. Then it can be shown that \eqref{strong_sv}, together with \eqref{p_g}, also implies that $\lambda$ obeys the slowly varying condition \eqref{s_v}\footnote{More precisely, from Section \ref{loc_reg_sct} we know that conditions \eqref{s_v} and the second inequality in \eqref{p_g} are equivalent under the assumptions \eqref{strong_sv} and \eqref{p_g_b}}. Thus the class of weight functions considered in this Section is a proper subclass of that considered in Section \ref{loc_reg_sct}. It is worthy to be noticed that weight functions described in the examples 1, 2, given in Section \ref{wf_sct}, are included in the class of weight functions that we are considering here, whereas the multi-quasi-elliptic weight function illustrated in the example 3 does not meet all the assumptions required here, precisely the sub-additivity ($\mathcal{SA}$) is not satisfied unless the complete polyhedron $\mathcal P$ gives rise to a quasi-homogeneous weight function of type \eqref{quasi_ell_wf}. Additional examples of weight functions obeying conditions ($\mathcal{SA}$), ($\mathcal{SH}$) and ($\mathcal{PG}$) are provided by the following
\begin{equation*}
\lambda_{r,s}(\xi)=\langle\xi\rangle^s\left[\log(2+\langle\xi\rangle)\right]^r\,,\quad\mbox{for}\,\, r,s\in ]0,+\infty[\,,
\end{equation*}
which were studied by Triebel \cite{TRI_3} (see also \cite{GM2}), or even by such functions as
\begin{equation*}
\begin{split}
\langle\xi\rangle^2_{\mu,\nu}&=1+\sum\limits_{j=1}^n\vert\xi_j\vert^{\mu_j}\left[\log(2+\vert\xi_j\vert)\right]^{\nu_j}\,,\\
&\quad\mbox{for}\,\,\mu=(\mu_1,\dots,\mu_n), \nu=(\nu_1,\dots,\nu_n)\in ]0,+\infty[^n\,,
\end{split}
\end{equation*}
or
\begin{equation*}
\Lambda_{s,\mathcal P}(\xi)=\langle\xi\rangle^s+\log(\lambda_{\mathcal P}(\xi))\,,\quad\mbox{for}\,\, s\in]0,+\infty[\,,
\end{equation*}
being $\lambda_{\mathcal P}(\xi)$ the multi-quasi-elliptic weight associated to a complete polyhedron $\mathcal P$, as it was introduced in Example 3 of Section \ref{wf_sct} (see \eqref{mqe_wf}).
}
\end{remark}

\smallskip
In order to take advantage of the slowly varying condition \eqref{s_v} (which allows in particular the symbolic calculus for smooth classes $S^r_{\rho, \lambda}(\Omega)$, see Section \ref{appl_sct}), it is convenient to introduce here another family of neighborhoods of an arbitrary set $X$ (in the frequency space $\mathbb R^n_{\xi}$), associated to the weight function $\lambda$, besides the $[\lambda]-$neighborhoods $X_{[\varepsilon\lambda]}$ already defined as in \eqref{omega_neighb}. For arbitrary $X\subset\mathbb R^n$ and $\varepsilon>0$ we set
\begin{equation}\label{omega_neighb_eucl}
X_{\varepsilon\lambda}:=\bigcup\limits_{\xi_0\in X}\left\{\xi\in\mathbb R^n\,:\,\,\vert\xi-\xi_0\vert<\varepsilon\lambda(\xi_0)^{1/\mu}\right\}\,,
\end{equation}
where $\mu>0$ is the same exponent involved in \eqref{p_g} (hence in \eqref{s_v} according to Remark \ref{rmk:15}); we will refer to the set $X_{\varepsilon\lambda}$ as the $\lambda-${\it neighborhood} of $X$ of size $\varepsilon$.

In the following for an open set $\Omega\subset\mathbb R^n$ and $x_0\in\Omega$, we also set for short $X_{\varepsilon\lambda}(x_0):=B_\varepsilon(x_0)\times X_{\varepsilon\lambda}$, where $B_\varepsilon(x_0)$ denotes the open ball in $\Omega$ centered at $x_0$ with radius $\varepsilon$.

\medskip
Compared to the case of $[\lambda]-$neighborhoods of a set $X$, to define the corresponding $\lambda-$neighborhoods the weight function $\lambda$ is replaced by the Euclidean norm, as the measure of the distance from points in $X_{\varepsilon\lambda}$ to points in $X$. This reflects into a slightly different behaviour of $\lambda-$neighborhoods: it is clear (just from the definition) that for $\varepsilon>0$ arbitrarily small the set $X_{\varepsilon\lambda}$ is never empty (unless $X=\emptyset$), cf. Remark \ref{rmk:intorni}; it is also clear that $X_{\varepsilon\lambda}$ is open, for it is the union of a family of open balls in $\mathbb R^n$ (centered at points of $X$).

The same set inclusions as given in Lemma \ref{mcl_lemma1} remain true also when the $[\lambda]-$neighborhoods of a set are replaced by the $\lambda-$neighborhoods, see \cite{RO1}, \cite{GM2.1} for the proof.
\begin{lemma}\label{mcl_lemma3}
Given $\varepsilon>0$, there exists $0<\varepsilon^\prime<\varepsilon$ such that for every $X\subset\mathbb R^n$
\begin{itemize}
\item[(1)] $\left(X_{\varepsilon^\prime\lambda}\right)_{\varepsilon^\prime\lambda}\subset X_{\varepsilon\lambda}$;
\item[(2)] $\left(\mathbb R^n\setminus X_{\varepsilon\lambda}\right)_{\varepsilon^\prime\lambda}\subset\mathbb R^n\setminus X_{\varepsilon^\prime\lambda}$.
\end{itemize}
\end{lemma}


A significant relation between $[\lambda]-$ and $\lambda-$neighborhoods is established by the next two results.
\begin{lemma}\label{lemma_inclusione_intorni}
Let $c>0$ be arbitrarily fixed. For every $\varepsilon>0$ there exists $0<\varepsilon^\prime<\varepsilon$ such that the set inclusion
\begin{equation}\label{inclusione_intorni}
\left(X\cap\left\{\lambda(\xi)>c/\varepsilon^\prime\right\}\right)_{\varepsilon^\prime\lambda}\subset X_{[\varepsilon\lambda]}\cap\left\{\lambda(\xi)>c/\varepsilon\right\}
\end{equation}
holds true for every $X\subset\mathbb R^n$.
\end{lemma}
\begin{proof}
Let $0<\varepsilon^\prime<\min\{1,\varepsilon\}$ be such that $X\cap\left\{\lambda(\xi)>c/\varepsilon^\prime\right\}$ be nonempty and take an arbitrary $\xi\in\left(X\cap\left\{\lambda(\xi)>c/\varepsilon^\prime\right\}\right)_{\varepsilon^\prime\lambda}$\footnote{If $X$ is unbounded then $X\cap\left\{\lambda(\xi)>c/\varepsilon^\prime\right\}\neq\emptyset$ for $\varepsilon^\prime>0$ arbitrarily small, because of the left inequality of \eqref{p_g}.}; then there exists some $\xi_0\in X$ such that
\begin{equation}\label{eq:1}
\vert\xi-\xi_0\vert<\varepsilon^\prime\lambda(\xi_0)^{1/\mu}\qquad\mbox{and}\qquad\lambda(\xi_0)>c/\varepsilon^\prime\,.
\end{equation}
From \eqref{p_g} and \eqref{eq:1} we get
\begin{equation}\label{eq:1.1}
\begin{split}
\lambda(\xi-\xi_0)&\le C(1+\vert\xi-\xi_0\vert)^\mu\le C2^{\mu-1}(1+\vert\xi-\xi_0\vert^\mu)\\
&<C2^{\mu-1}(1+\varepsilon^{\prime\,\mu}\lambda(\xi_0))<C2^{\mu-1}(\varepsilon^\prime/c\lambda(\xi_0)+\varepsilon^{\prime\,\mu}\lambda(\xi_0))\\
&<C2^{\mu-1}\varepsilon^\prime(1/c+1)\lambda(\xi_0)\,,
\end{split}
\end{equation}
hence $\lambda(\xi-\xi_0)<\varepsilon\lambda(\xi_0)$ provided that $\varepsilon^\prime$ is such that
\begin{equation*}
C2^{\mu-1}\varepsilon^\prime(1/c+1)<\varepsilon\,.
\end{equation*}
Thus $\xi\in X_{[\varepsilon\lambda]}$ provided that $0<\varepsilon^\prime<\min\left\{1,\frac{\varepsilon}{C2^{\mu-1}(1/c+1)}\right\}$.

Let us now prove that $\lambda(\xi)>c/\varepsilon$ up to a further shrinking of $\varepsilon^\prime$. We use again conditions ($\mathcal{SA}$), ($\mathcal{SH}$), ($\mathcal{PG}$) and \eqref{eq:1.1} to find
\begin{equation*}
\begin{split}
\lambda(\xi)&\ge 1/C\lambda(\xi_0)-\lambda(\xi-\xi_0)\ge 1/C\lambda(\xi_0)-C(1+\vert\xi-\xi_0\vert)^\mu\\
&\ge 1/C\lambda(\xi_0)-C2^{\mu-1}(1+\vert\xi-\xi_0\vert^\mu)>1/C\lambda(\xi_0)-C2^{\mu-1}(1+\varepsilon^{\prime\,\mu}\lambda(\xi_0))\\
&=\left(1/C-C2^{\mu-1}\varepsilon^{\prime\,\mu}\right)\lambda(\xi_0)-C2^{\mu-1}\,,
\end{split}
\end{equation*}
from which we deduce, using also \eqref{eq:1},
\begin{equation*}
\lambda(\xi)>\frac1{2C}\lambda(\xi_0)-C2^{\mu-1}>\frac{c}{2C\varepsilon^\prime}-C2^{\mu-1}>\frac{c}{4C\varepsilon^\prime}>\frac{c}{\varepsilon}\,,
\end{equation*}
provided that $\varepsilon^\prime>0$ is chosen such that
\begin{equation*}
\varepsilon^\prime<\min\left\{\frac{1}{2C^{2/\mu}},\frac{c}{2^{\mu+1}C^2},\frac{\varepsilon}{4C}\right\}\,.
\end{equation*}
This ends the proof that $\xi\in X_{[\varepsilon\lambda]}\cap\left\{\lambda(\xi)>c/\varepsilon\right\}$.
\end{proof}
\begin{remark}\label{rmk:intorni1}
{\rm If $X$ is bounded, the set $X\cap\left\{\lambda(\xi)>c/\varepsilon^\prime\right\}$ (hence the neighborhood $\left(X\cap\left\{\lambda(\xi)>c/\varepsilon^\prime\right\}\right)_{\varepsilon^\prime\lambda}$) is empty for $\varepsilon^\prime>0$ sufficiently small, thus the inclusion \eqref{inclusione_intorni} becomes trivial. However, thanks to \eqref{inclusione_intorni}, this never occurs when $X$ is unbounded; in such a case the set $X\cap\left\{\lambda(\xi)>c/\varepsilon^\prime\right\}$ is nonempty for arbitrarily small $\varepsilon^\prime>0$, since $\lambda$ is unbounded on $X$ as a consequence of the left inequality in \eqref{p_g}. This yields in particular that, for an unbounded set $X$ the $[\lambda]-$neighborhood $X_{[\varepsilon\lambda]}$ is nonempty with size $\varepsilon>0$ arbitrarily small, cf. Remark \ref{rmk:intorni}.}
\end{remark}
\begin{corollary}\label{cor_intorni}
For every $\varepsilon>0$ there exists $0<\varepsilon^\prime<\varepsilon$ such that for all $X\subset\mathbb R^n$
\begin{equation}\label{inclusione_intorni_2}
\left(X_{[\varepsilon^\prime\lambda]}\right)_{\varepsilon^\prime\lambda}\subset X_{[\varepsilon\lambda]}\,.
\end{equation}
\end{corollary}
\begin{proof}
 In view of Lemma \ref{mcl_lemma1} we first notice that for arbitrary $\varepsilon>0$ we may find $0<\varepsilon^\ast<\varepsilon$ sufficiently small such that
 \begin{equation*}
 \left(X_{[\varepsilon^\ast\lambda]}\right)_{[\varepsilon^\ast\lambda]}\subset X_{[\varepsilon\lambda]}\,.
 \end{equation*}
Then combining the results of Lemma \ref{mcl_lemma1} and Lemma \ref{lemma_inclusione_intorni}, with $X_{[\varepsilon^\ast\lambda]}$ instead of $X$, another $0<\varepsilon^\prime<\varepsilon^\ast$ sufficiently small can be chosen such that
\begin{equation*}
\begin{split}
\left(X_{[\varepsilon^\prime\lambda]}\right)&_{\varepsilon^\prime\lambda}\equiv\left(X_{[\varepsilon^\prime\lambda]}\cap\{\lambda(\xi)>\hat c/\varepsilon^\prime\}\right)_{\varepsilon^\prime\lambda}\\
&\subset\left(X_{[\varepsilon^\ast\lambda]}\cap\{\lambda(\xi)>\hat c/\varepsilon^\prime\}\right)_{\varepsilon^\prime\lambda}\subset \left(X_{[\varepsilon^\ast\lambda]}\right)_{[\varepsilon^\ast\lambda]}\subset X_{[\varepsilon\lambda]}\,,
\end{split}
\end{equation*}
where $\hat c>0$ is given in Lemma \ref{mcl_lemma1}. The proof is complete.
\end{proof}
In order to perform the subsequent analysis, the next technical lemma will be useful; for its proof, the reader is addressed to \cite[Lemma 1.10]{RO1}, see also \cite[Lemma 1]{GM2.1}.
\begin{proposition}\label{mcl_prop_1}
For arbitrary $\varepsilon>0$ and $X\subset\mathbb R^n$ there exists a symbol $\sigma=\sigma(\xi)\in S^0_{\lambda}$ such that ${\rm supp}\,\sigma\subset X_{\varepsilon\lambda}$ and $\sigma(\xi)=1$ if $\xi\in X_{\varepsilon^\prime\lambda}$, for a suitable $\varepsilon^\prime>0$, with $0<\varepsilon^\prime<\varepsilon$, depending only on $\varepsilon$ and $\lambda$. Moreover for every $x_0\in\Omega$, where $\Omega\subset\mathbb R^n$ is an open set, there exists a symbol $\tau_0(x,\xi)\in S^0_{\lambda}(\Omega)$ such that ${\rm supp}\,\tau_0\subset X_{\varepsilon\lambda}(x_0)$ and $\tau_0(x,\xi)=1$, for $(x,\xi)\in X_{\varepsilon^\ast\lambda}(x_0)$, with a suitable $\varepsilon^\ast$ satisfying $0<\varepsilon^\ast<\varepsilon$.
\end{proposition}

\begin{remark}\label{rmk:17}
{\rm As an application of Corollary \ref{cor_intorni}, one can easily see that a statement similar to Proposition \ref{mcl_prop_1} also holds when $\lambda-$neighborhoods are replaced with the corresponding $[\lambda]-$neighborhoods; indeed for arbitrary $X\subset\mathbb R^n$ and $\varepsilon>0$, take $0<\tilde\varepsilon<\varepsilon$ such that $\left(X_{[\tilde\varepsilon\lambda]}\right)_{\tilde\varepsilon\lambda}\subset X_{[\varepsilon\lambda]}$ and apply the result of Proposition \ref{mcl_prop_1}, where $X$ is replaced by $X_{[\tilde\varepsilon\lambda]}$. Then some numbers $0<\varepsilon^{\prime\prime}<\varepsilon^\prime<\tilde\varepsilon$ and a symbol  $\sigma=\sigma(\xi)\in S^0_{\lambda}$ exist such that ${\rm supp}\,\sigma\subset \left(X_{[\tilde\varepsilon\lambda]}\right)_{\varepsilon^\prime\lambda}\subset \left(X_{[\tilde\varepsilon\lambda]}\right)_{\tilde\varepsilon\lambda}\subset X_{[\varepsilon\lambda]}$ and $\sigma\equiv 1$ on $\left(X_{[\tilde\varepsilon\lambda]}\right)_{\varepsilon^{\prime\prime}\lambda}$ (hence on $X_{[\tilde\varepsilon\lambda]}$). As for the construction of a counterpart of the variable coefficients symbol $\tau_0(x,\xi)\in  S^0_{\lambda}(\Omega)$ in the second part of the statement above, it comes from the use of the symbol $\sigma(\xi)$ by following the same lines as in Proposition \ref{mcl_prop_1}, see \cite[Lemma 1]{GM2.1}.}
\end{remark}

\begin{definition}\label{mcl_ell}
Let us consider a symbol $a(x,\xi)\in S^r_{\rho,\lambda}(\Omega)$, $x_0\in\Omega$ and $X\subset\mathbb R^n$. We say that $a(x,\xi)$ (or the corresponding pseudodifferential operator) is microlocally $[\lambda]-$elliptic in $X$ at point $x_0$, writing $a(x,\xi)\in{\rm mce}_{r,[\lambda]}X(x_0)$, if there exist constants $c_0>0$ and $\varepsilon>0$ sufficiently small such that
\begin{equation}\label{mcl_ell_ineq}
\vert a(x_0,\xi)\vert\ge c_0\lambda(\xi)^r\,,\qquad\mbox{for}\,\,\xi\in X_{[\varepsilon\lambda]}\,.
\end{equation}
\end{definition}
\begin{remark}\label{rmk:16}
{\rm Let us remark that in the above definition we do not explicitly require that frequencies $\xi$, for which \eqref{mcl_ell_ineq} holds true, are larger than some positive constant (that is usual when defining an ellipticity condition, cf. \eqref{ell}); indeed, because of Lemma \ref{mcl_lemma1}, $\xi\in X_{[\varepsilon\lambda]}$ yields $\lambda(\xi)>\hat c/\varepsilon$ and, for sufficiently small $\varepsilon>0$, the latter turns out to be a largeness condition on $\xi$, in view of the polynomial growth condition ($\mathcal{PG}$).}
\end{remark}
Let us recall the following notion, providing a microlocal counterpart of the notion of regularizing symbol
\begin{definition}
We say that a symbol $a(x,\xi)\in S^r_{\rho,\lambda}(\Omega)$ is rapidly decreasing in $\Theta\subset \Omega\times{\bf R}^n$ if there exists $a_0(x,\xi)\in S^r_{\rho,\lambda}(\Omega)$ such that $a(x,\xi)-a_0(x,\xi)\in S^{-\infty}(\Omega)$ and  $a_0(x,\xi)=0$ in $\Theta$.
\end{definition}
The following notion is a natural substitute of that of {\it characteristic set} of a symbol, in the absence of any homogeneity property.
\begin{definition}
We define the characteristic filter of a symbol $a(x,\xi)\in S^r_{\rho,\lambda}(\Omega)$ at a point $x_0\in\Omega$ to be the set
\begin{equation}\label{char_filter}
\Sigma_{[\lambda], x_0}a:=\left\{ X\subset\mathbb R^n\,:\,\,\, a(x,\xi)\in{\rm mce}_{r,[\lambda]}(\mathbb R^n\setminus X)(x_0)\right\}\,.
\end{equation}
\end{definition}
Using Lemma \ref{mcl_lemma1}, it is easy to check that $\Sigma_{[\lambda], x_0}a$ is a $[\lambda]-$filter.

\smallskip
The reader is addressed to \cite{RO1} and \cite{GM3} where analogous notions as above are stated in a more general setting.

\smallskip
Arguing on the properties of $\lambda-$neighborhoods of a set and the slowly varying condition \eqref{s_v} as in the proof of \cite[Lemma 4.3]{GM3}, one can prove that $a(x,\xi)\in S^s_{\rho,\lambda}(\Omega)$ is microlocally $[\lambda]-$elliptic in $X$ at point $x_0$ if and only if
\begin{equation}\label{mcl_ell_ineq_BIS}
\vert a(x,\xi)\vert\ge c^\ast\lambda(\xi)^r\,,\qquad\mbox{for}\,\,(x,\xi)\in B_{\tilde\varepsilon}(x_0)\times \left(X_{[\tilde\varepsilon\lambda]}\right)_{\tilde\varepsilon\lambda}\,,
\end{equation}
for suitable constants $c^\ast>0$ and sufficiently small $\tilde\varepsilon>0$.

Then following the same lines of the proof of \cite[Theorem 4.6]{GM3} one can prove the following
\begin{proposition}\label{mcl_parametrix}
For every symbol $a(x,\xi)\!\in\! S^r_{\rho,\lambda}(\Omega)$ microlocally $[\lambda]\!-$elliptic in $\{x_0
\}\times X$, there exists a symbol $b(x,\xi)\in S^{-r}_{\rho,\lambda}(\Omega)$ such that the associated operator $b(x,D)$ is properly supported and
\begin{equation}\label{MPTEO1a}
b(x,D)a(x,D)={\rm Id}+c(x,D),
\end{equation}
where $c(x,\xi)\in S^0_{\rho,\lambda}(\Omega)$ is rapidly decreasing in $B_{\tilde\varepsilon}(x_0)
\times \left(X_{[\tilde\varepsilon\lambda]}\right)_{\tilde\varepsilon\lambda}$ for a suitable small $\tilde\varepsilon>0$.
\end{proposition}

\medskip
For $s\in\mathbb R$, $p\in[1,+\infty]$, $U$ open neighborhood of $x_0\in\mathbb R^n$ and $X\subset\mathbb R^n$ given, let $\mathcal FL^p_{s,\lambda,{\rm loc}}(U)$ and $\mathcal FL^p_{s,\lambda,{\rm mcl}}(x_0\times X)$ denote the local and microlocal Fourier Lebesgue classes corresponding to the weight function $\lambda^s$, according to Definitions  \ref{loc_w_FL_spaces}, \ref{micro_FL_space}. Agreeing with these notations, we denote by $\Xi_{[\lambda],\mathcal FL^p_{s,\lambda},x_0}$ the related $[\lambda]-$filter of Fourier Lebesgue singularities.

\medskip
By resorting to Proposition \ref{mcl_prop_1} and arguing similarly as in the proof of \cite[Proposition 4.5]{GM2.2} and \cite[Proposition 4.10]{GM3}, we are able to prove the following characterization of microlocal Fourier Lebesgue spaces.

\begin{proposition}\label{char_FL}
Let $x_0\in\Omega$, $\Omega\subset\mathbb R^n$ open set, and $X\subset\mathbb R^n$ be given. A distribution $u\in\mathcal D^\prime(\mathbb R^n)$ belongs to $\mathcal FL^p_{s,\lambda,{\rm mcl}}(x_0\times X)$ if and only if one of the following two equivalent conditions is satisfied:
\begin{itemize}
\item[i.] there exist constants $0<\varepsilon^\prime<\varepsilon$ sufficiently small and $\phi\in C^\infty_0(\Omega)$, with $\phi(x_0)\neq 0$, such that
    \begin{equation}\label{char_eq1}
    \sigma(D)(\phi u)\in \mathcal FL^p_{s,\lambda}(\mathbb R^n)\,,
    \end{equation}
    where $\sigma=\sigma(\xi)\in S^0_{\lambda}$ is some symbol satisfying ${\rm supp}\,\sigma\subset X_{[\varepsilon\lambda]}$ and $\sigma\equiv 1$ on $X_{[\varepsilon^\prime\lambda]}$;
\item[ii.] There exist an operator $\tau(x,D)\in \widetilde{\rm Op}\,S^0_{\lambda}(\Omega)$ microlocally $[\lambda]-$elliptic in $\{x_0\}\times X$ such that
    \begin{equation}\label{char_eq2}
    \tau(x,D)u\in\mathcal FL^p_{s,\lambda,{\rm loc}}(\Omega)\,.
    \end{equation}
\end{itemize}
\end{proposition}
Following similar arguments to those in \cite[Propositions 5.1, 5.2]{GM2.2} we give the following results.
\begin{proposition}\label{cont1}
Let $s\in\mathbb R$, $r>0$, $x_0\in\Omega$, $X\subset\mathbb{R}^n$, $a(x,D)\in\widetilde{{\rm Op}}S^r_{\rho,\lambda}(\Omega)$ be given. Then for $p\in[1,\infty]$ and $u\in mcl\mathcal FL^p_{s,\lambda}(x_0\times X)$ one has $a(x,D)u\in mcl\mathcal FL^{p}_{s-r,\lambda}(x_0\times X)$.
\end{proposition}

\begin{proposition}\label{cont2}
For $s\in\mathbb R$, $r>0$, $x_0\in\Omega$, $X\subset\mathbb{R}^n$, let $a(x,D)\in\widetilde{\rm Op}S^r_{\rho,\Lambda}(\Omega)$ be microlocally $[\lambda]-$elliptic in $\{x_0\}\times X$. Then for every $p\in[1,\infty]$ and $u\in\mathcal{D}'(\mathbb R^n)$ such that $a(x,D)u\in mcl\mathcal FL^{p}_{s-r,\lambda}(x_0\times X)$ one has $u\in mcl\mathcal FL^{p}_{s}(x_0\times X)$.
\end{proposition}

It is also straightforward to show that the results of Propositions \ref{cont1}, \ref{cont2} can be restated in terms of the filters of Fourier Lebesgue singularities and characteristic filter of a symbol as follows.
\begin{proposition}\label{filter_inclusions}
Let $s\in\mathbb R$, $r>0$ be arbitrary real numbers, $a(x,D)\in\widetilde{\rm Op}\,S_{\rho, \lambda}(\Omega)$, $x_0\in\Omega$ and $p\in[1,\infty]$. Then the following inclusions are satisfied for every $u\in\mathcal{D}'(\mathbb R^n)$:
\begin{equation*}
\Xi_{[\lambda],\mathcal FL^p_{s-r,\lambda},x_0}a(x,D)u\cap\Sigma_{[\lambda],x_0}a\subset \Xi_{[\lambda],\mathcal FL^p_{s,\lambda},x_0}u\subset\Xi_{[\lambda],\mathcal FL^p_{s-r,\lambda},x_0}a(x,D)u\,.
\end{equation*}
\end{proposition}
\subsection{Semilinear equations}\label{semi-linear_sct}
Gathering the results collected in the preceding Sections \ref{w_FL_pdo_sct} and \ref{mcl_reg_sct}, we prove here a result of microlocal regularity in Fourier Lebesgue spaces for solutions to semilinear partial differential equations of type \eqref{semilin_eqt} already considered in Section \ref{loc_reg_sct}. Throughout this Section, we assume that $a(x,D)$, $b_i(x,D)$ for $1\le i\le M$ in \eqref{semilin_eqt} are properly supported operators with symbols $S^r_\lambda(V_{x_0})$ and $S^{r-\varepsilon}_\lambda(V_{x_0})$ on some open bounded neighborhood $V_{x_0}$ of a point $x_0$, where as in Section \ref{w_FL_pdo_sct} $0<\varepsilon<r$ are given, and the nonlinear function $F=F(x,\zeta)$ depending as a $C^\infty-$function on its first argument of $x\in V_{x_0}$ and entire analytic on its second argument $\zeta=(\zeta_i)_{1\le i\le M}\in\mathbb C^M$, satisfying the same requirement made in Section \ref{algebra_sct} (see also Remark \ref{rmk:14}).

The following result was originally proved in \cite{G}.
\begin{theorem}\label{semilin_thm}
For $0<\varepsilon<r$ and $1\le p\le +\infty$ given as above, let $\tau$, $\tilde t$, $s$ be positive real numbers such that
\begin{equation}\label{condizioni_esponenti}
\tau+r-\varepsilon\le\tilde t<s
\end{equation}
and $\lambda^{-\tau}\in L^q(\mathbb R^n)$, where $q$ is the conjugate exponent of $p$. As for the semilinear equation \eqref{semilin_eqt}, let us assume that the pseudodifferential operator $a(x,D)$ is $[\lambda]-$microlocally elliptic in $X\subset\mathbb R^n$ at the point $x_0$ and the source term $f=f(x)$ belongs to $\mathcal FL^p_{s-r,\lambda,{\rm mcl}}(x_0\times X)$. Then every solution $u\in\mathcal FL^p_{\tilde t,\lambda,{\rm loc}}(x_0)$ to the equation \eqref{semilin_eqt} with source term $f$ also satisfies
\begin{equation}\label{reg_finale}
u\in\mathcal FL^p_{t,\lambda,{\rm mcl}}(x_0\times X)\,,\quad\mbox{for all}\,\, t\le\min\left\{s,\tilde t+\left(E\left(\frac{\tilde t-r-\tau}{\varepsilon}\right)+2\right)\varepsilon\right\}\,,
\end{equation}
where $E(\theta)$ is the greatest integer less than or equal to $\theta\in\mathbb R$.
\end{theorem}
\begin{proof}
The proof relies on a bootstrapping argument similar to the one used to prove Proposition \ref{reg_prop1}. So let $u\in\mathcal FL^p_{\tilde t,\lambda,{\rm loc}}(x_0)$ be a solution to equation \eqref{semilin_eqt}. From Propositions \ref{continuity3} we get
\begin{equation}\label{bi}
b_i(x,D)u\in \mathcal FL^p_{\tilde t-r+\varepsilon,\lambda,{\rm loc}}(x_0)\,,\quad i=1,\dots,M\,.
\end{equation}
In view of the assumptions \eqref{condizioni_esponenti}, $\lambda^{-\tau}\in L^q(\mathbb R^n)$ and the sub-additivity of $\lambda$ we may apply the result of Theorem \ref{continuity2} and its consequences stated in Remark \ref{rmk:14}; with reference to the statement of that theorem, here $\lambda^{\tau}$ plays the role of the weight functions $\sigma$ whereas $\lambda^{\tilde t-r+\varepsilon}$ plays the role of both the weight functions $\lambda$, $\Lambda\,\,\,$\footnote{Notice in particular that from the sub-additivity of $\lambda$ and $\lambda^{\tau}\preceq\lambda^{\tilde t-r+\varepsilon}$ (following from \eqref{condizioni_esponenti}), we derive that $\lambda^{\tilde t-r+\varepsilon}$ satisfies condition \eqref{B_q} with the conjugate exponent of $p$ (that is required to apply Theorem \ref{continuity2}.}. Thus it follows from \eqref{bi} that
\begin{equation}\label{Fbi}
F(x,b_i(x,D)u)_{1\le i\le M}\in \mathcal FL^p_{\tilde t-r+\varepsilon,\lambda,{\rm loc}}(x_0)\,.
\end{equation}
If $s\le\tilde t+\varepsilon$, we derive from \eqref{semilin_eqt} that $a(x,D)u\in\mathcal FL^p_{s-r,\lambda,{\rm mcl}}(x_0\times X)$ hence $u\in\mathcal FL^p_{s,\lambda,{\rm mcl}}(x_0\times X)$ in view of the $[\lambda]-$microellipticity of $a(x,D)$ in $X$ at point $x_0$ and Proposition \ref{cont2} (notice that $\tilde t+\varepsilon\le 2\tilde t-r-\tau+2\varepsilon$ under the assumption \eqref{condizioni_esponenti} then $s\le \tilde t+\varepsilon$ implies $s=\min\{s,2\tilde t-r-\tau+2\varepsilon\}$); if on the contrary $s>\tilde t+\varepsilon$ again from \eqref{semilin_eqt} we derive that $a(x,D)u\in\mathcal FL^p_{\tilde t-r+\varepsilon,\lambda,{\rm mcl}}(x_0\times X)$ hence $u\in\mathcal FL^p_{\tilde t+\varepsilon,\lambda,{\rm mcl}}(x_0\times X)$ by the same arguments as before. In this latter case, using once again Propositions \ref{continuity3} and \ref{cont1} we get
\begin{equation*}
b_i(x,D)u\in \mathcal FL^p_{\tilde t-r+\varepsilon,\lambda,{\rm loc}}(x_0)\cap \mathcal FL^p_{\tilde t-r+2\varepsilon,\lambda,{\rm mcl}}(x_0\times X)\,,\quad 1\le i\le M\,.
\end{equation*}
Now we would like to apply Theorem \ref{continuity2} where the role of the weight functions $\sigma$, $\lambda$ and $\Lambda$ is covered respectively by $\lambda^\tau$, $\lambda^{\tilde t-r+\varepsilon}$ and $\lambda^{\tilde t-r+2\varepsilon}$; the only assumption to check is $\Lambda\preceq\lambda^2/\sigma$ which amounts to have that $\tilde t\ge \tau+r$ (cf. \eqref{cndt:2}). If this is the case, Theorem \ref{continuity2} applies to find
\begin{equation*}
F(x,b_i(x,D)u)\in \mathcal FL^p_{\tilde t-r+\varepsilon,\lambda,{\rm loc}}(x_0)\cap \mathcal FL^p_{\tilde t-r+2\varepsilon,\lambda,{\rm mcl}}(x_0\times X)
\end{equation*}
hence from \eqref{semilin_eqt} and Proposition \ref{cont2}
\begin{equation*}
a(x,D)u\in\mathcal FL^p_{s-r,\lambda,{\rm mcl}}(x_0\times X)\,\,\Rightarrow\,\,u\in\mathcal FL^p_{s,\lambda,{\rm mcl}}(x_0\times X) \,,\quad\mbox{if}\,\,s\le\tilde t+2\varepsilon
\end{equation*}
or
\begin{equation*}
a(x,D)u\in\mathcal FL^p_{\tilde t-r+2\varepsilon,\lambda,{\rm mcl}}(x_0\times X)\,\,\Rightarrow\,\,u\in\mathcal FL^p_{\tilde t+2\varepsilon,\lambda,{\rm mcl}}(x_0\times X) \,,\quad\mbox{otherwise}.
\end{equation*}
In the latter case
\begin{equation*}
b_i(x,D)u\in \mathcal FL^p_{\tilde t-r+\varepsilon,\lambda,{\rm loc}}(x_0)\cap \mathcal FL^p_{\tilde t-r+3\varepsilon,\lambda,{\rm mcl}}(x_0\times X)\,,\quad 1\le i\le M
\end{equation*}
and, provided that $\tilde t\ge\tau+r+\varepsilon$, we are still in the position to apply Theorem \ref{continuity2} where $\lambda^{\tau}$ and $\lambda^{\tilde t-r+\varepsilon}$ play again the role of $\sigma$ and $\lambda$, while $\lambda^{\tilde t-r+3\varepsilon}$ plays the role of $\Lambda$. For $\tilde t\ge\tau+r+\varepsilon$ Theorem \ref{continuity2} yields
\begin{equation*}
F(x,b_i(x,D)u)\in \mathcal FL^p_{\tilde t-r+\varepsilon,\lambda,{\rm loc}}(x_0)\cap \mathcal FL^p_{\tilde t-r+3\varepsilon,\lambda,{\rm mcl}}(x_0\times X)
\end{equation*}
and again
\begin{equation*}
a(x,D)u\in\mathcal FL^p_{s-r,\lambda,{\rm mcl}}(x_0\times X)\,\,\Rightarrow\,\,u\in\mathcal FL^p_{s,\lambda,{\rm mcl}}(x_0\times X) \,,\quad\mbox{if}\,\,s\le\tilde t+3\varepsilon
\end{equation*}
or
\begin{equation*}
a(x,D)u\in\mathcal FL^p_{\tilde t-r+3\varepsilon,\lambda,{\rm mcl}}(x_0\times X)\,\,\Rightarrow\,\,u\in\mathcal FL^p_{\tilde t+3\varepsilon,\lambda,{\rm mcl}}(x_0\times X) \,,\quad\mbox{otherwise}.
\end{equation*}
By an iteration of the above procedure we find that if the integer $j\ge 0$ is such that
\begin{equation}\label{condizione_j}
\tilde t<\tau+r+j\varepsilon
\end{equation}
then
\begin{equation*}
u\in\mathcal FL^p_{\min\{\tilde t+(j+1)\varepsilon,s\},\lambda,{\rm mcl}}(x_0\times X)\,.
\end{equation*}
The smallest nonnegative integer $j$ satisfying \eqref{condizione_j} is $\tilde j=E\left(\frac{\tilde t-r-\tau}{\varepsilon}\right)+1$ (from \eqref{condizioni_esponenti} $E\left(\frac{\tilde t-\tau-r}{\varepsilon}\right)\ge -1$ follows), hence $\tilde t+(\tilde j+1)\varepsilon=\tilde t+\left(E\left(\frac{\tilde t-r-\tau}{\varepsilon}\right)+2\right)\varepsilon$. This gives
\begin{equation*}
u\in\mathcal FL^p_{\min\left\{\tilde t+\left(E\left(\frac{\tilde t-r-\tau}{\varepsilon}\right)+2\right)\varepsilon,s\right\},\lambda,{\rm mcl}}(x_0\times X)\,,
\end{equation*}
which completes the proof. 
\end{proof}
In terms of the filter of Fourier Lebesgue singularities and characteristic filter of a symbol, the result of Theorem \ref{semilin_thm} can be restated as follows: for every solution $u\in\mathcal FL^p_{\tilde t,\lambda,{\rm loc}}(x_0)$ to equation \eqref{semilin_eqt} one has
\begin{equation*}
\Xi_{[\lambda],\mathcal FL^p_{s-r,\lambda},x_0}f\cap\Sigma_{[\lambda],x_0}a\subset \Xi_{[\lambda],\mathcal FL^p_{t,\lambda},x_0}u\,,
\end{equation*}
for all $t\le\min\left\{s,\tilde t+\left(E\left(\frac{\tilde t-r-\tau}{\varepsilon}\right)+2\right)\varepsilon\right\}$.
\begin{remark}\label{rmk:18}
{\rm Because of the lower estimate of condition ($\mathcal{PG}$), a sufficient condition for $\lambda^{-\tau}\in L^q(\mathbb R^n)$
is $\tau>\frac{n}{\nu q}$.}
\end{remark}

\subsection{The case of quasi-homogeneous equations}\label{q_hom_sct}
In this section, we deal with pseudodifferential operators whose smooth symbols are associated to a quasi-homogeneous weight as defined in the Example 2 of Section \ref{wf_sct}. We recall that for $M=(m_1,\dots,m_n)\in\mathbb N^n$, with $m_\ast:=\min\limits_{1\le j\le n}m_j\ge 1$, the {\it quasi-homogeneous weight} is defined as
\begin{equation}\label{q_hom_w_n}
\langle\xi\rangle_M:=\left(1+\sum\limits_{j=1}^n\xi_j^{2m_j}\right)^{1/2}\,.
\end{equation}
Throughout the rest of this Section, we will make use of the following notations. We set  $m^\ast:=\max\limits_{1\le j\le n}m_j$, $\frac 1M:=\left( \frac{1}{m_1},\dots, \frac{1}{m_n}\right)$ and defined the {\it quasi-homogeneous norm} as
\begin{equation}\label{q_hom_norm}
\vert\xi\vert^2_M:=\sum\limits_{j=1}^n\xi_j^{2m_j}\,.
\end{equation}
Clearly the usual Euclidean norm $\vert\xi\vert$ corresponds to the quasi-homogeneous norm in the case of $M=(1,\dots,1)$. For every $\alpha\in\mathbb Z^n_+$, $
\xi\in\mathbb R^n$ and $t>0$ we also set $\langle\alpha,\frac1{M}\rangle:=\sum\limits_{j=1}^n\frac{\alpha_j}{m_j}$ and $t^{1/M}\xi:=(t^{1/m_1}\xi_1,\dots, t^{1/m_n}\xi_n)$. It is worth to notice that, in spite of the terminology, the quasi-homogeneous norm $\vert\cdot\vert_M$ is not a norm; instead of the homogeneity and the triangle inequality, required for norms, the  quasi-homogeneous norm enjoys the following properties:
\begin{itemize}
\item[(i)] {\it Quasi-Homogeneity}: for all $t>0$, $\xi\in\mathbb{R}^n$
\begin{equation*}
\vert t^{ 1/M}\xi\vert_M =t\vert\xi\vert_M\,;
\end{equation*}
\item[(ii)] {\it Sub-additivity}: a constant $C\ge 1$ depending only on $M$ exists such that
\begin{equation*}
\vert\xi+\eta\vert_M\leq C(\vert \xi\vert_M+\vert \eta \vert_M)\,,\qquad \forall\,\xi,\eta\in\mathbb R^n\,.
\end{equation*}
\end{itemize}
For $R>0$ and $x_0\in\mathbb R^n$, the {\it $M-$open ball} centered at $x_0$ with radius $R$ is defined to be the set
\begin{equation*}
B_M(x_0;R);=\left\{x\in\mathbb R^n\,:\,\,\vert x-x_0\vert_M<R\right\}\,.
\end{equation*}
The set
\begin{equation}\label{M-unit-ball}
\mathbb S_M:=\left\{x\in\mathbb R^n\,:\,\,\vert x\vert_M=1\right\}
\end{equation}
is the {\it unit $M-$sphere} (centered at the origin). For further details and properties of quasi-homogeneous norm and weight, we address the reader to \cite{GMqh1}, \cite{GMqh2}.

According to the behavior of the weight \eqref{q_hom_w_n} expressed by the estimates \eqref{q_hom_der}, we introduce suitable classes of smooth symbols displaying a decaying behavior of {\it quasi-homogeneous} type.


\begin{definition}\label{def:1}
Given $r\in\mathbb{R}$, $S^r_{M}$ will be the class of functions $a(x,\xi)\in
C^{\infty}(\mathbb{R}^{2n})$ such that for all multi-indices $\alpha,\beta\in\mathbb{Z}^n_+$ there exists $C_{\alpha,\beta}>0$ such that:
\begin{equation}\label{eqh:1}
\vert \partial^{\beta}_x\partial^{\alpha}_{\xi}a(x,\xi)\vert\le
C_{\alpha,\beta}\langle\xi\rangle_M^{r-\langle\alpha,\frac1{M}\rangle},\quad\forall x,\,\xi\in\mathbb{R}^n\,.
\end{equation}
If $\Omega$ is an arbitrary open subset of $\mathbb R^n$, we denote by $S^r_M(\Omega)$ the local class of functions $a(x,\xi)\in C^\infty(\Omega\times\mathbb R^n)$ such that $\phi(x)a(x,\xi)\in S^r_M$ for all $\phi\in C^\infty_0(\Omega)$.
\end{definition}

\medskip
Due to the underlying quasi-homogeneous structure, in the present framework the whole theory of propagation of singularities can be based upon a suitable notion of ``conical'' set in frequency space adapted to this structure.

Let us recall below some basic notions, see \cite{GMqh2} for more details. Later on it is set for short $T^{\circ}\mathbb{R}^n:=\mathbb R^n\times(\mathbb R^n\setminus\{0\})$.
\begin{definition}\label{Mcone}
We say that a set $\Gamma\subset\mathbb{R}^n\setminus\{0\}$ is an $M-$cone (or is $M-$conic), if
\begin{equation*}
\xi\in\Gamma\quad\Rightarrow\quad t^{1/M}\xi\in\Gamma\,\,,\,\,\forall\,t>0\,.
\end{equation*}
\end{definition}

For $\eta\in\mathbb R^n$ and $R>0$ the set
\begin{equation}\label{M-cone}
\Gamma_M(\eta;R):=\left\{t^{1/M}\xi\,:\,\,\xi\in B_M(\eta;R)\,,\,\,t>0\right\}\cap(\mathbb R^n\setminus\{0\})
\end{equation}
is $M-$conic; it is called the {\it $M-$cone generated by $B_M(\eta;R)$}.

\medskip
Since \eqref{q_hom_w_n} also belongs to the class of weight functions considered in Sections \ref{loc_reg_sct}, \ref{mcl_reg_sct}, the results considered there, based upon the notion of $[\lambda]-$filter, could be applied to the quasi-homogeneous setting (that is $\lambda(\xi)=\langle\xi\rangle_M$).
The next results of this Section will provide some evidences that these two alternative approaches are essentially equivalent.
\begin{proposition}\label{equivalence_prop}
There exist constants $\hat c>0$ and $\hat\varepsilon>0$ sufficiently small such that for all $0<\varepsilon\le\hat\varepsilon$ another $0<\varepsilon^\prime<\varepsilon$ exists such that for all $M-$conic set $X\subset\mathbb R^n\setminus\{0\}$
\begin{equation}\label{equiv:1}
X_{[\varepsilon^\prime\langle\cdot\rangle_M]}\subset\bigcup\limits_{\eta\in X\cap\mathbb S_M}\Gamma_M(\eta;\varepsilon)\cap\left\{\langle\xi\rangle_M>\hat c/\varepsilon\right\}
\end{equation}
and, conversely,
\begin{equation}\label{equiv:2}
\bigcup\limits_{\eta\in X\cap\mathbb S_M}\Gamma_M(\eta;\varepsilon^\prime)\cap\left\{\langle\xi\rangle_M>\hat c/\varepsilon^\prime\right\}\subset X_{[\varepsilon\langle\cdot\rangle_M]}\,.
\end{equation}
\end{proposition}
\begin{proof}
For a $M-$conic set $X$ and an arbitrary $\tilde\varepsilon>0$, take $\xi\in X_{[\tilde\varepsilon\langle\cdot\rangle_M]}$ and let $\eta\in X$ such that
\begin{equation}\label{equiv_eq:1}
\langle\xi-\eta\rangle_M<\tilde\varepsilon\langle\eta\rangle_M\,.
\end{equation}
Making use of the trivial inequalities
\begin{equation*}
1/\sqrt 2(1+\vert\zeta\vert_M)\le\langle\zeta\rangle_M\le 1+\vert\zeta\vert_M\,,\qquad\forall\,\zeta\in\mathbb R^n\,,
\end{equation*}
\eqref{equiv_eq:1} implies
\begin{equation*}
1+\vert\xi-\eta\vert_M<\sqrt 2\tilde\varepsilon (1+\vert\eta\vert_M)\,,
\end{equation*}
hence
\begin{equation}\label{equiv_eq:2}
\vert\xi-\eta\vert_M<\sqrt 2\tilde\varepsilon\vert\eta\vert_M
\end{equation}
provided that $0<\tilde\varepsilon<1/\sqrt 2$. Since in particular \eqref{equiv_eq:1} implies $\eta\neq 0$, because of the $M-$homogeneity of the $M-$norm, condition \eqref{equiv_eq:2} can be reformulated as follows
\begin{equation*}
\begin{split}
\vert\xi-\eta\vert_M &<\sqrt 2\tilde\varepsilon\vert\eta\vert_M\quad\Leftrightarrow\quad \vert\vert\eta\vert_M^{1/M}(\zeta-\tilde\eta)\vert_M=\vert\eta\vert_M\vert\zeta-\tilde\eta\vert_M<\sqrt 2\tilde\varepsilon\vert\eta\vert_M\\
&\Leftrightarrow\quad\vert\zeta-\tilde\eta\vert_M<\sqrt 2\tilde\varepsilon\,,
\end{split}
\end{equation*}
where $\zeta:=\vert\eta\vert_M^{-1/M}\xi$ and $\tilde\eta:=\vert\eta\vert_M^{-1/M}\eta\in\mathbb S_M\cap X$ because $X$ is $M-$conic. The last inequality above means that $\zeta$ belongs to the $M-$open ball centered at $\tilde\eta$ with radius $\sqrt 2\tilde\varepsilon$, thus $\xi=\vert\eta\vert_M^{1/M}\zeta\in\Gamma_M(\tilde\eta;\sqrt 2\tilde\varepsilon)$ cf. \eqref{M-cone}. Since in view of Lemma \ref{mcl_lemma1}, $\eta\in X_{[\tilde\varepsilon\langle\cdot\rangle_M]}$ also implies
\begin{equation*}
\langle\eta\rangle_M>\hat c/\tilde\varepsilon\,,
\end{equation*}
for suitable $\hat c>0$ independent of $X$ and $\tilde\varepsilon>0$, the inclusion \eqref{equiv_eq:1} follows taking $\hat{\varepsilon}=\frac{1}{\sqrt 2 C}$ and choosing for each $0<\varepsilon\le\hat\varepsilon$, $0<\varepsilon^\prime\le\varepsilon/\sqrt 2$.

Conversely, let $\xi\in \bigcup\limits_{\eta\in X\cap\mathbb S_M}\Gamma_M(\eta;\tilde\varepsilon)\cap\left\{\langle\xi\rangle_M>\hat c/\tilde\varepsilon\right\}$, where $\tilde\varepsilon>0$ is still arbitrary and is chosen sufficiently small such that
\begin{equation*}
1+\vert\xi\vert_M\ge\langle\xi\rangle_M>\hat c/\tilde\varepsilon\quad\Rightarrow\quad\vert\xi\vert_M>\hat c/\tilde\varepsilon-1\ge\hat c/(2\tilde\varepsilon)\,,
\end{equation*}
that is $0<\tilde\varepsilon\le\hat c/2$. By definition (see \eqref{M-cone}) there exist $t>0$ and $\tilde\eta\in X\cap\mathbb S_M$ such that
\begin{equation*}
\xi=t^{1/M}\zeta\,\,\,\,\,\mbox{for some}\,\,\zeta\in B_M(\tilde\eta;\tilde\varepsilon)\,.
\end{equation*}
Therefore, in view of the $M-$homogeneity, we get
\begin{equation*}
\vert\xi-t^{1/M}\tilde\eta\vert_M=\vert t^{1/M}\zeta-t^{1/M}\tilde\eta\vert_M=t\vert \zeta-\tilde\eta\vert_M<\tilde\varepsilon t=\tilde\varepsilon t\vert\tilde\eta\vert_M=\tilde\varepsilon\vert t^{1/M}\tilde\eta\vert_M\,.
\end{equation*}
Since $X$ is $M-$conic, $\eta:=t^{1/M}\tilde\eta\in X$; hence with such an $\eta\in X$ we have
\begin{equation}\label{equiv_eq:3}
\vert\xi-\eta\vert_M<\tilde\varepsilon\vert\eta\vert_M\,.
\end{equation}
On the other hand, from the sub-additivity (ii), $\vert\xi\vert_M>\hat c/(2\tilde\varepsilon)$ and \eqref{equiv_eq:3} we derive
\begin{equation}\label{equiv_eq:4}
\vert\eta\vert_M\ge\frac1 C\vert\xi\vert_M-\vert\xi-\eta\vert_M\ge \left(\frac1C-\tilde\varepsilon\right)\vert\xi\vert_M\ge\frac1{2C}\vert\xi\vert_M\ge\frac{\hat c}{2C\tilde\varepsilon}\,,
\end{equation}
provided that $0<\tilde\varepsilon\le\frac1{2C}$. Summing up \eqref{equiv_eq:3}, \eqref{equiv_eq:4} then gives
\begin{equation*}
\frac{\hat c}{2C}+\vert\xi-\eta\vert_M<2\tilde\varepsilon\vert\eta\vert_M<2\tilde\varepsilon\langle\eta\rangle_M\,.
\end{equation*}
Combining the latter inequality with
\begin{equation*}
\frac{\hat c}{2C}+\vert\xi-\eta\vert_M\ge (1+\vert\xi-\eta\vert_M)\min\left\{1,\frac{\hat c}{2C}\right\}\ge\frac1{\sqrt 2}\min\left\{1,\frac{\hat c}{2C}\right\}\langle\xi-\eta\rangle_M
\end{equation*}
we finally obtain
\begin{equation*}
\langle\xi-\eta\rangle_M<\widehat C\tilde\varepsilon\langle\eta\rangle_M\,,
\end{equation*}
with a suitable constant $\widehat C>0$ independent of $\tilde\varepsilon$, hence $\xi\in X_{[\widehat C\tilde\varepsilon\langle\cdot\rangle_M]}$. From the previous argument, we conclude that the second inclusion \eqref{equiv:2} holds true.
\end{proof}
Essentially the result above tells that a $[\langle\cdot\rangle_M]-$neighborhood of a $M$-conic set $X$ is made by an arbitrary union of open $M$-cones ``outgoing from points of $X\cap\mathbb S_M$, truncated near their vertex".

\begin{remark}\label{rmk:19}
{\rm It is worthwhile noticing that the quasi-homogeneous symbols considered in Definition \ref{def:1} are related to the weighted smooth symbols introduced in Section \ref{appl_sct} by the following inclusion
\begin{equation*}
S^r_M\subset S^r_{1/m_\ast,\langle\cdot\rangle_M}
\end{equation*}
(a similar inclusion being valid for the corresponding classes of local symbols).}
\end{remark}

\subsection{ Example}\label{EX}
For $M=(1,2)$, let us consider in $\mathbb R^2$  the quasi-homogeneous weight function
\begin{equation}\label{q_h_2}
\langle \xi\rangle_{M}=\left(1+\xi_1^2+\xi^4_2 \right)^{1/2}\,.
\end{equation}
We introduce the following operator
\begin{equation}
\begin{split}
P(x,\partial)= x_1\partial_{x_1}+ i\partial_{x_1}-\partial_{x_2^2}^2\,.
\end{split}
\end{equation}

Its symbol $P(x,\xi)=ix_1\xi_1-\xi_1+\xi_2^2$ belongs to the local class $S^1_M(\Omega)$ where $\Omega=\mathbb R^2$.

Introducing the {\it $M$-characteristic set} of $P(x,\partial)$ as
\begin{equation}\label{eqEX4}
\textup{Char}P=\left\{ (x,\xi)\in\mathbb R^2_x\times\mathbb R^2_\xi\setminus\{0\}\, ,\, P(x,\xi)=0\right\},
\end{equation}
we have
\begin{equation*}
\textup{Char}P=\left\{ (0,x_2,\xi_1,\xi_2);\, x_2\in\mathbb R\,,\,\,\xi_1=\xi_2^2\,,\,\,\xi_2\neq 0 \right\}\!=\!\{0\}\times\mathbb R\times\!\!\bigcap\limits_{0<k<1}\!\!(\mathbb R^2\setminus X_k)\,,
\end{equation*}
where
\begin{equation}\label{eqEX2.1}
X_k=\left\{(\xi_1,\xi_2)\in \mathbb R^2\,\,;\,\, \xi_1\le(1-k)\xi_2^2\,\,\textup{or}\,\, \xi_1\ge\frac{1}{1-k}\xi_2^2 \right\}\,,\,\,0<k<1\,.
\end{equation}
Notice also that $P(x,\xi)$ is {\it quasi-homogeneous of degree one}, in the sense that
\begin{equation}\label{QH}
P(x,t^{1/M}\xi)=P(x,t\xi_1,t^{1/2}\xi_2)=tP(x,\xi)\,,\qquad\forall\,t>0\,.
\end{equation}
The properties collected above yield that the symbol $P(x,\xi)$ is $\langle\cdot\rangle_M-$elliptic at every point $x^0=(x^0_1,x^0_2)\in\mathbb R^2$ with $x^0_1\neq 0$; indeed since $\vert P\vert$ does not vanish at each point of the compact set $\{x^0\}\times\mathbb S_M$, being $\mathbb S_M=\{\eta=(\eta_1,\eta_2)\in\mathbb R^2\,:\,\,\eta_1^2+\eta_2^4=1\}$ is the unit $M-$sphere, by continuity
\begin{equation*}
c_0:=\inf\limits_{\eta\in\mathbb S_M}\vert P(x^0,\eta)\vert>0\,.
\end{equation*}
Hence the quasi-homogeneity of $P$ yields for $\vert\xi\vert_M\ge 1$:
\begin{equation*}
\vert P(x^0,\xi)\vert=\vert\xi\vert_M\vert P(x^0,\eta)\vert\ge c_0\vert\xi\vert_M\ge c_0/\sqrt 2\langle\xi\rangle_M\,,
\end{equation*}
where $\eta:=\vert\xi\vert_M^{1/M}\xi\in\mathbb S_M$.

\medskip

By resorting to Proposition \ref{equivalence_prop}, we show now that at every point $x^0=(0,x^0_2)$, with an arbitrary $x^0_2\in\mathbb R$, $P(x,\xi)$ is $[\langle\cdot\rangle_M]-$microlocally elliptic in any set of the family $\{X_k\}_{0<k<1}$ defined by \eqref{eqEX2.1}. So, let us take arbitrary $x^0=(0,x^0_2)$ and $0<k<1$; since $\vert P\vert$ is different from zero and continuous (hence uniformly continuous) on the compact subset $\{x^0\}\times (X_k\cap\mathbb S_M)$ of $T^\circ\mathbb R^n\setminus\textup{Char}P$, some constants $c_k>0$ and $0<\tilde\varepsilon<1$ sufficiently small can be found such that
\begin{equation}
\vert P(x^0,\eta)\vert\ge c_k\,,
\end{equation}
for $\eta$ ranging on the covering of $X_k\cap\mathbb S_M$ made by the open $M-$balls $B_M(\tilde\eta;\tilde\varepsilon)$ centered at points $\tilde\eta$ of $X_k\cap\mathbb S_M$ with radius $\tilde\varepsilon$. Take now an arbitrary point $\xi\in\bigcup\limits_{\tilde\eta\in X_k\cap\mathbb S_M}\Gamma_M(\tilde\eta;\tilde\varepsilon)$ such that $\vert\xi\vert_M>c/\tilde\varepsilon$ with suitable $c>0$; then $\tilde\eta\in X_k\cap\mathbb S_M$ and $t>0$ exist such that $\xi=t^{1/M}\eta$ for some $\eta\in B_M(\tilde\eta;\tilde\varepsilon)$. Since $\vert\tilde\eta\vert_M=1$, we may take $\tilde\varepsilon$ so small that $\vert\eta\vert_M\le\hat c$ for some positive constant $\hat c$ (independent of $\eta$ and $\tilde\eta$). Exploiting again the quasi-homogeneity of $P$ and the quasi-norm $\vert\cdot\vert_M$ we get
\begin{equation*}
\vert P(x^0,\xi)\vert=t\vert P(x^0,\eta)\vert\ge c_kt=\frac{c_k}{\hat c}t\vert\eta\vert_M=\frac{c_k}{\hat c}\vert\xi\vert_M\ge\tilde c_k\langle\xi\rangle_M\,,
\end{equation*}
with suitable $\tilde c_k>0$. Since the set $X_k$ is $M-$conic, in view of Proposition \ref{equivalence_prop} there exists $0<\varepsilon^\prime<\tilde\varepsilon$ (up to shrink $\tilde\varepsilon$ if necessary) such that
\begin{equation*}
(X_k)_{[\varepsilon^\prime\langle\cdot\rangle_M]}\subset\bigcup\limits_{\tilde\eta\in X_k\cap\mathbb S_M}\Gamma_M(\tilde\eta;\tilde\varepsilon)\cap\left\{\vert\xi\vert_M>c/\tilde\varepsilon\right\}\,.
\end{equation*}
This shows that $P$ is microlocally $[\langle\cdot\rangle_M]-$elliptic in $X_k$ at the point $x^0=(0,x^0_2)$.

\medskip
Since $P(x,\xi)\in S^1_M(\mathbb R^2)$ and in view of Remark \ref{rmk:19}, the results of propagation of Fourier Lebesgue singularities for linear an semilinear equations, collected in the preceding Sections \ref{mcl_reg_sct}, \ref{semi-linear_sct}, can be applied to the operator $P(x,\partial)$.

\medskip
Let $u\in\mathcal D^\prime(\mathbb R^2)$ be a solution to the linear equation
\begin{equation}\label{lin_P}
P(x,\partial)u=f(x)\,,
\end{equation}
with a given forcing term $f$. Applying to \eqref{lin_P} the result of Proposition \ref{filter_inclusions} (with $r=1$), we obtain at once that the following inclusions
\begin{equation*}
\Xi_{[\langle\cdot\rangle_M],\,\mathcal FL^p_{s-1,M},\,x^0}f\cap\Sigma_{[\langle\cdot\rangle_M],\,x^0}P\subset\Xi_{[\langle\cdot\rangle_M],\,\mathcal FL^p_{s,M},\,x^0}u\subset\Xi_{[\langle\cdot\rangle_M],\,\mathcal FL^p_{s-1,M},\,x^0}f
\end{equation*}
hold true for all $s\in\mathbb R$ and $p\in[1,+\infty]$.

\medskip
Consider now the following semilinear equation
\begin{equation}\label{eq_semilineare}
P(x,\partial)u+F(x,u,\partial_{x_2}u)=f(x)\,,
\end{equation}
where $F=F(x,\zeta)$ is a nonlinear function of $x=(x_1,x_2)$ and $\zeta=(\zeta_1,\zeta_2)$ fulfilling the regularity assumptions stated in Theorem \ref{semilin_thm}, and $f=f(x)$ some given forcing term. With respect to the quasi-homogeneous weight \eqref{q_h_2}, the order of derivatives of the unknown function $u$ involved in the nonlinearity in \eqref{eq_semilineare} is easily seen to be $\le 1/2$ (that is such derivatives are properly supported operators in $S_M^{l}$ with order $l\le 1/2$). Then we may apply to \eqref{eq_semilineare} the result of Theorem \ref{semilin_thm} (with $r=1$ and $\varepsilon=1/2$) to prove the following statement.
\begin{proposition}
Given $x^0=(x^0_1,x^0_2)\in\mathbb R^2$, $p\in[1,+\infty]$ and $s>\tilde t>\frac{2}{q}+\frac12$, with $\frac1{p}+\frac1{q}=1$, let  $u\in\mathcal FL^p_{\tilde t,\,M,\,{\rm loc}}(x^0)$ be a solution to \eqref{eq_semilineare}.
\begin{itemize}
\item[a.] If $2\tilde t-2-\frac{4}{q}\notin\mathbb Z$ then
\begin{equation}\label{filtro_2}
\Xi_{[\langle\cdot\rangle_M],\,\mathcal FL^p_{s-1,M},\,x^0}f\cap\Sigma_{[\langle\cdot\rangle_M],\,x^0}P\subset\Xi_{[\langle\cdot\rangle_M],\,\mathcal FL^p_{t,M},\,x^0}u\,,
\end{equation}
holds true for all $t\le\min\left\{s,\tilde t+1+\frac12 E\left(2\tilde t-2-\frac{4}{q}\right)\right\}$;
\item[b.] if $2\tilde t-2-\frac{4}{q}\in\mathbb Z$ then the inclusion \eqref{filtro_2} holds true for all $t\le\min\left\{s,\tilde t+\frac12+\frac12 E\left(2\tilde t-2-\frac{4}{q}\right)\right\}$.
\end{itemize}
\end{proposition}
\begin{proof}
For any $\tilde t>\frac{2}{q}+\frac12$, let $\tau>0$ be chosen such that $\tilde t-\frac12\ge\tau>\frac{2}{q}$, then Theorem \ref{semilin_thm} can be directly applied to equation \eqref{eq_semilineare} (where, according to the observations above, it is set $r=1$, $\varepsilon=1/2$, and we also make use of Remark \ref{rmk:18} for $\lambda=\langle\cdot\rangle_M$ and $\tau$ as above) to find that inclusion \eqref{filtro_2} holds true for all $t\le\min\left\{s,\tilde t+1+\frac12E(2\tilde t-2-2\tau)\right\}$. To conclude, it is enough to observe that for $2\tilde t-2-\frac{4}{q}\notin\mathbb Z$ we can take $\tau$ sufficiently close to $\frac{2}{q}$ to have
\begin{equation*}
E\left(2\tilde t-2-2\tau\right)=E\left(2\tilde t-2-\frac{4}{q}\right)\footnote{According to the result of Theorem \ref{semilin_thm}, this corresponds to the best possible range for $t$.};
\end{equation*}
this proves the statement {\it a}.

If, on the contrary, one has $2\tilde t-2-\frac{4}{q}\in\mathbb Z$ then $2\tilde t-2-2\tau<2\tilde t-2-\frac{4}{q}=E(2\tilde t-2-\frac{4}{q})$ whenever $\tau$ is taken as above; then for $\tau$ sufficiently close to $\frac{2}{q}$ we get
\begin{equation*}
E(2\tilde t-2-2\tau)=E\left(2\tilde t-2-\frac{4}{q}\right)-1
\end{equation*}
which gives the statement {\it b}.
\end{proof}

{\small
Gianluca Garello\\
Dipartimento di Matematica\\
Universit\`a di Torino\\
Via Carlo Alberto 10, I-10123 Torino, Italy\\
gianluca.garello@unito.it}
\medskip

\noindent {\small
Alessandro Morando\\
DICATAM-Sezione di Matematica\\
Universit\`a di Brescia\\
Via Valotti 9, I-25133 Brescia, Italy\\
alessandro.morando@unibs.it}

\end{document}